\documentclass[12pt]{amsart}
\usepackage[dvipsnames]{xcolor}
\usepackage{amsfonts,amssymb,amsmath,amscd,amstext}
\usepackage[colorlinks=true,linkcolor=teal,citecolor=purple]{hyperref}
\usepackage[utf8]{inputenc}
\usepackage{graphicx}
\usepackage{changes}
\usepackage{verbatim}
\usepackage[sort&compress]{natbib}
\usepackage{comment}
\usepackage{dsfont}
\usepackage{esint}
\usepackage[noabbrev,capitalize,nameinlink]{cleveref}
\usepackage{mathdots}
\usepackage[a4paper,top=2.3cm,bottom=2.3cm,left=2cm,right=2cm]{geometry}
\renewcommand{\leq}{\leqslant}
\renewcommand{\geq}{\geqslant}
\renewcommand{\le}{\leqslant}

\newcommand{\hh}{{\mathbb{H}}}

\newcommand{\eps}{\varepsilon}

\newcommand{\G}{\mathbb G}
\newcommand{\R}{\mathbb R}

\definechangesauthor[color=red]{J}
\definecolor{champagne}{rgb}{0.97, 0.91, 0.81}

\definecolor{asparagus}{rgb}{0.53, 0.66, 0.42}

\DeclareMathOperator{\spann}{span}

\newtheorem{theorem}{Theorem}[section]
\newtheorem{proposition}[theorem]{Proposition}
\newtheorem{lemma}[theorem]{Lemma}
\newtheorem{corollary}[theorem]{Corollary}

\theoremstyle{definition}
\newtheorem{definition}[theorem]{Definition}

\newtheorem{remark}[theorem]{Remark}
\newtheorem{example}[theorem]{Example}

\newcommand{\V}{{\mathbb{V}}}

\theoremstyle{remark}

\numberwithin{equation}{section}

\definechangesauthor[color=purple]{S}
\definechangesauthor[color=red]{P}

\title{Characterizations of Sobolev and BV functions on Carnot groups} 

\author{Francesco Serra Cassano}
\address{Dipartimento di Matematica, Università degli Studi di Trento, Via Sommarive, 14, 38123 Povo (Trento), Italia}
\email{francesco.serracassano@unitn.it}
\author{Kilian Zambanini}
\address{Dipartimento di Matematica, Università degli Studi di Trento, Via Sommarive, 14, 38123 Povo (Trento), Italia}
\email{kilian.zambanini@unitn.it}

\begin{document}\begin{abstract}
    We establish two characterizations of real-valued Sobolev and BV functions on Carnot groups. The first is obtained via a nonlocal approximation of the distributional horizontal gradient, while the second is based  on an $L^p$ Taylor approximation, in the spirit of the results by Bourgain, Brezis and Mironescu \cite{BBM}. 
\end{abstract}
\maketitle

\section{Introduction} 
The main purpose of this paper is to provide two characterizations of Sobolev and bounded variation functions $f:\G\to\R$ defined on a Carnot group $\G$ (see Definitions \ref{horSobspace} and \ref{bvG}). 
The first characterization is obtained by approximating the so-called \emph{weak horizontal gradient} $\nabla_\G f$ (see Section 2 for its definition) by means of suitable \emph{nonlocal horizontal gradients} (see Theorem \ref{firstcharacterization} below). The second is based on an approximation of $f$ by means of a first-order horizontal Taylor polynomial (see Theorem \ref{secondcharBVSobfunct} below).

The first characterization is in the spirit of the celebrated paper by Bourgain, Brezis, and Mironescu \cite{BBM}. In particular, we adopt the approach recently developed in the Euclidean setting by Brezis and Mironescu \cite{brezis} to recover a well-known result due to Mengesha and Spector (see \cite{mengesha} and the references therein for a comprehensive account of this topic).

The second characterization extends to Carnot groups an interesting result of Spector, who, in the Euclidean setting, related the notion of $L^p$-differentiability introduced by Calderón and Zygmund \cite{calderon61} with characterizations of Sobolev spaces via the asymptotic behaviour of nonlocal functionals from \cite{BBM}. This led to a characterization of Sobolev and bounded variation functions in terms of a first-order Taylor approximation (see \cite{spector,Spector16} and the references therein).

One of the aims of the present work is to provide, as far as possible, a unified approach to the characterization of Sobolev and bounded variation function spaces within the framework of the theory of Bourgain, Brezis, and Mironescu (commonly referred to as the BBM formula), a broad topic that has been extensively studied by many authors, while focusing in particular on the setting of Carnot groups. These structures provide a significant class of metric measure spaces where techniques from Euclidean or Riemannian geometry cannot always be  applied directly.

The definition and main properties of Carnot groups are collected in Section 2. We briefly recall that a Carnot group $\G$ is a simply connected nilpotent Lie group that can be represented as $\G \equiv (\R^n, \cdot)$ and is equipped with:
\begin{itemize}
\item a family of intrinsic dilations $(\delta_\lambda)_{\lambda>0} : \G \to \G$, which are automorphisms of $\G$;
\item a subbundle $H\G$ of the tangent bundle $T\G$, called the horizontal subbundle, which generates the whole $T\G$ via commutators;
\item a homogeneous norm $N:\G\to [0,\infty)$, which induces a distance $d$ on $\G$ that is left-invariant and homogeneous with respect to the intrinsic dilations.
\end{itemize}
The Haar measure of $\G$ coincides with the $n$-dimensional Lebesgue measure $\mathcal L^n$ on $\R^n$, and the metric (Hausdorff) dimension of $(\G,d)$ coincides with its homogeneous dimension $Q$. If $\G$ is commutative, then $\G \equiv \R^n$ as a Euclidean vector space, otherwise $Q > n$. In particular, the homogeneous dimension $Q$ plays the same role as the topological dimension $n$ in the Euclidean setting.

Let us now introduce the relevant approximating (nonlocal) functionals considered throughout the paper. Let $(\rho_\eps)_\eps$ be a family of functions on $\G$ (which we call \textit{mollifiers}) satisfying the following properties, which are standard assumptions introduced in \cite{BBM} by Bourgain, Brezis and Mironescu:
\begin{equation}\label{P1}\tag{P1}
\rho_\eps(x)\geq 0,\quad \rho_\eps\in L^1(\G);
\end{equation}
\begin{equation}\label{P2}\tag{P2}
\rho_\eps(x)=\tilde\rho_\eps(N(x))\quad\text{for some Borel function }\tilde \rho_\eps:[0,+\infty)\to[0,+\infty),
\end{equation}
and we refer to $\tilde\rho_\eps$ as the \emph{profile} of $\rho_\eps$;
\begin{equation}\label{P3}\tag{P3}
\int_\G \rho_\eps(h)\,dh=1\quad\text{for every }\eps>0;
\end{equation}
\begin{equation}\label{P4}\tag{P4}
\lim_{\eps\to 0}\int_{\G\setminus B(0,\delta)}\rho_\eps (h)\,dh=0\quad\text{for all } \delta>0,
\end{equation}
where $B(x,r)=B_N(x,r)$ denotes the open ball with respect to the distance $d$ induced by $N$.

Following Brezis and Mironescu \cite{brezis}, given a sequence $(\rho_\varepsilon)_\varepsilon$ as above, we define, for any $f\in L^1_{\rm loc}(\G)$, the associated \emph{nonlocal horizontal gradients} by
\begin{equation}\label{Veps}
V_\varepsilon(f)(x):= Q\int_\G \frac{f(x\cdot h)-f(x)}{N(h)}\,\nabla_\G N(h)\,\rho_\varepsilon(h)\,dh,
\end{equation}
for every $x\in\G$ for which the above integral is well defined.
We also introduce the functional
\begin{equation}\label{tildeVeps}
\widetilde V_\varepsilon(f)(x):=\int_\G\frac{|f(x\cdot h)-f(x)|}{N(h)}\,|\nabla_\G N(h)|\,\rho_\varepsilon(h)\,dh,
\end{equation}
which is well defined for every $x\in\G$ (possibly taking the value $+\infty$).

Two other relevant functionals studied here are the energy-type functionals naturally associated with $V_\eps(f)$. For $1\leq p<+\infty$ and $\eps>0$, we set
\begin{equation}\label{nonlocal energies}
I_{\eps, p}(f):=\int_\G\int_\G\frac{|f(x\cdot h)-f(x)|^p}{N(h)^p}|\nabla_\G N(h)|^p\,\rho_\eps(h)\,dh\,dx,
\end{equation}
\begin{equation}\label{IepspBarb}
I^*_{\eps, p}(f):=\int_\G\int_\G\frac{|f(x\cdot h)-f(x)|^p}{N(h)^p}\,\rho_\eps(h)\,dh\,dx.
\end{equation}

Nonlocal energies of the form \eqref{IepspBarb} have been extensively studied in the literature. In their seminal paper \cite{BBM}, Bourgain, Brezis, and Mironescu considered the Euclidean setting, namely the case in which $\mathbb{G}=\mathbb{R}^n$ and $N$ denotes the standard Euclidean norm. 
They proved that the quantities $I^*_{\varepsilon,p}(f)$ are uniformly bounded whenever $f \in W^{1,p}(\Omega)$ or $f \in BV(\Omega)$ (with $p=1$ in the latter case) and that, in addition, the converse implication also holds. Moreover, 
they established the convergence
\[
I^*_{\eps,p}(f)\to C_{n,p}\|\nabla f\|^p_{L^p(\Omega)} \quad \text{as } \eps\to 0,
\]
for every $f \in W^{1,p}(\Omega)$, where $C_{n,p}$ is a positive constant depending only on the dimension $n$ and on the exponent $p$. Later, Dávila \cite[Theorem 1]{davila} extended this result to functions of bounded variation. 
In the Euclidean case, the horizontal gradient $\nabla_\G\equiv\nabla$ and
\begin{equation}\label{eikonaleq}
|\nabla_\G N(x)|=1 \quad \text{for each } x\in\G\setminus\{0\}.
\end{equation}
If \eqref{eikonaleq} holds, then the functionals $I_{\eps,p}$ and $I^*_{\eps,p}$ coincide. More generally, the eikonal equation \eqref{eikonaleq} is satisfied almost everywhere in a general Carnot group $\G$ for the norm $N_c$ induced by the Carnot--Carathéodory distance (see Remark \ref{eikeqccnor}). However, in general, \eqref{eikonaleq} may fail for a homogeneous norm $N$ on a Carnot group $\G$, and the functionals $I^*_{\eps,p}$ and $I_{\eps,p}$ need not be equivalent (see Remark~\ref{comparIIstar}). 

Several generalizations of these results in Euclidean spaces have been obtained; see, for instance, \cite{ABBF1,ABBF2,Ponce2004,LeoniSpector1,LeoniSpector2,ludwig,NguPinSquaVe,brezis-nguyen,brezis-nguyen2,brezis-nguyen3,brezis-nguyen4,PSV19,brezis-nguyen5,BVSY,BSVSY1,BSVSY2,Ponce,CS1,CS2,BCCS, LMP, nguyensquassina, bonicatto,BSY, bourgainnguyen, nguyen1,nguyen2, nguyen3} and the references therein.

In the setting of Carnot groups, related results have also been obtained; see, for instance, Barbieri \cite{barbieri}, Garofalo--Tralli \cite{GT,GT2}, Maalaoui--Pinamonti \cite{MP}, Zhang--Zhu \cite{tongJie} and the references therein.
 In particular, under the additional assumption of $N$ being invariant under horizontal rotations, Barbieri \cite[Proposition 3.5 and Theorem 3.6]{barbieri} established an analogous characterization for Sobolev functions in $W^{1,p}_\G(\G)$, $1<p<+\infty$, and proved that 
\begin{equation}\label{barbieribbm}
I^*_{\eps,p}(f)\to C_{Q,p}\|\nabla_\G f\|^p_{L^p(\Omega)} \quad \text{as } \eps\to 0
\end{equation}
where $C_{Q,p}$ is a positive constant depending only on $Q,p$ and the chosen homogeneous norm $N$.
More recently, results in this spirit have also been established in metric measure spaces; see, e.g., Brena--Pasqualetto--Pinamonti \cite{BPP}, Di Marino--Squassina \cite{DMS19}, Górny \cite{Gorny}, Bang-Xian--Xu--Zhu \cite{HXZ25}, Lahti--Pinamonti--Zhou \cite{lathi,LPZ}, Munnier \cite{Munnier}, Shimizu \cite{shimizu} and the references therein.

As is known in the Euclidean case (\cite[Theorem 1.1]{mengesha} or \cite[Proposition 1.7]{brezis}), we show (Corollary \ref{approximation}) that a similar convergence result holds for the nonlocal gradients, namely
\[
V_\eps(f)\to \nabla_\G f \quad \text{in } L^p(\G) \quad \text{if } f\in W^{1,p}_\G(\G).
\]
The analogous statement holds for maps in $BV_\G(\G)$ (see \cite[Theorem 1.2]{mengesha} or \cite[Proposition 1.8]{brezis} in the Euclidean case). We stress that this result, in contrast to that of Barbieri, does not require the rotational invariance hypothesis on the norm. 
Moreover, being proved for a general homogeneous norm $N$, it appears to be new even in the Euclidean case (see Remark~\ref{gradeuclidapprox}). 
{On the other hand, the validity of a similar convergence result for arbitrary norms in Euclidean spaces is due to Ludwig \cite{ludwig} in the case of the nonlocal energies $I_{\eps, p}^*(f)$. We recover this fact as a consequence of Theorem \ref{secondcharBVSobfunct} (see Corollary \ref{thmLudwig}).}
\medskip

We are now in a position to state our first characterization of Sobolev and BV functions on Carnot groups, which involves the functionals $V_\eps$, $\widetilde V_\eps$, $I_{\eps,p}$, and $I^*_{\eps,p}$.

\begin{theorem}[Characterization of Sobolev and BV functions by nonlocal gradients]\label{firstcharacterization}\phantom{}\\
Let $\G$ be a Carnot group endowed with a homogeneous norm $N$. Let $(\rho_\varepsilon)_\varepsilon$ be a sequence of mollifiers satisfying assumptions $\eqref{P1}\div\eqref{P4}$, and let $f\in L^p(\G)$, with $1\le p<\infty$.
\begin{itemize}
\item[(i)] If $1<p<\infty$, then
\begin{equation}\label{firstchari1}
f\in W^{1,p}_\G(\G)\Longleftrightarrow \limsup_{\eps\to 0}I_{\eps,p}(f)<+\infty;
\end{equation}
\begin{equation}\label{firstchari1bis}
f\in W^{1,p}_\G(\G)\Longleftrightarrow \limsup_{\eps\to 0}I^*_{\eps,p}(f)<+\infty;
\end{equation}
\begin{equation}\label{firstchari2}
f\in W^{1,p}_\G(\G)\Longleftrightarrow \limsup_{\eps\to 0}\|\widetilde V_\eps(f)\|_{L^p(\G)}<+\infty.
\end{equation}

In these cases, $V_\eps(f)$ is well defined and
\[
V_\eps(f)\to \nabla_\G f \quad \text{in } L^p(\G)\quad \text{as } \eps\to 0.
\]

\item[(ii)] If $p=1$, then
\begin{equation}\label{firstchari1BV}
f\in BV_\G(\G)\Longleftrightarrow \limsup_{\eps\to 0}I_{\eps,1}(f)<+\infty;
\end{equation}
\begin{equation}\label{firstchari1BVbis}
f\in BV_\G(\G)\Longleftrightarrow \limsup_{\eps\to 0}I^*_{\eps,1}(f)<+\infty;
\end{equation}
\begin{equation}\label{firstchari2BV}
f\in BV_\G(\G)\Longleftrightarrow \limsup_{\eps\to 0}\|\widetilde V_\eps(f)\|_{L^1(\G)}<+\infty.
\end{equation}

In this case, $V_\eps(f)$ is well defined and
\[
V_\eps(f)\rightharpoonup D_\G f \quad \text{in } \mathcal M(\G,\R^{m_1}) \quad \text{as } \eps\to 0.
\]
\end{itemize}
\end{theorem}

Equivalences \eqref{firstchari1bis} and \eqref{firstchari1BVbis} were recently proved by Lahti, Pinamonti, and Zhou \cite[Theorem 1.1]{lathi} in abstract metric measure spaces, which include Carnot groups, under an additional growth condition on the mollifiers. Our approach does not require any further assumptions on the family $(\rho_\eps)_\eps$.

The first key ingredient in the proof of the previous characterization is a representation formula for the nonlocal gradient $V_\varepsilon(f)$, valid for functions $f\in W_\G^{1,p}(\G)$ or $f\in BV_\G(\G)$ (see Theorem \ref{convolution}). This representation formula extends a deep result of Brezis and Mironescu \cite{brezis} from the Euclidean setting to Carnot groups. It allows us to exploit techniques from harmonic analysis on Carnot groups (see, for instance, \cite{folland,stein}), which naturally extend those from the Euclidean setting.

The second key ingredient is a careful approximation of the mollifiers $(\rho_\eps)_\eps$ by a family $(\bar\rho_\eps)_\eps$ satisfying, in addition to properties \eqref{P1}$\div$\eqref{P4}, the property of being compactly supported in $\G\setminus\{0\}$. This ensures that the nonlocal gradients corresponding to $\bar\rho_\eps$ are immediately well-defined for any $f\in L^1_{\rm loc}(\G)$, and the characterization can be obtained by estimating only the nonlocal gradient $V_\eps(f)$ instead of the potentially larger $\widetilde V_\eps(f)$ (see \eqref{P5} and Corollary \ref{characterization}).

The second main result of the paper concerns the characterization of a function $f\in W^{1,p}_\G(\G)$ (with $1\le p<\infty$), or $f\in BV_\G(\G)$, in terms of the notion of $L^p$-differentiability in the sense of Calderón and Zygmund (see Definition \ref{Lpdiff}). It is well known that, in the Euclidean setting, $L^p$-differentiability almost everywhere is only a necessary condition for $f$ to belong to $W^{1,p}(\R^n)$. A natural question is how to strengthen the notion of $L^p$-differentiability in order to obtain a characterization of Sobolev spaces (see \cite{Spector16} and the references therein for a comprehensive discussion).

To this end, our characterization reads as follows in the setting of Carnot groups.

\begin{theorem}[Characterization of Sobolev and $BV$ functions by $L^p$-Taylor approximation]\label{secondcharBVSobfunct}
Let $\G$ be a Carnot group equipped with a homogeneous norm $N$. Let $(\rho_\varepsilon)_\varepsilon$ be a sequence of mollifiers satisfying assumptions $\eqref{P1}\div\eqref{P4}$, and let $f\in L^p(\G)$, with $1\le p<\infty$. Then
\[
\begin{split}
f\in W^{1,p}_\G(\G)
\Longleftrightarrow\;
& f \text{ satisfies a first-order $L^p$-Taylor approximation with respect to } (\rho_\eps)_\eps, \\
& \text{that is, }
\lim_{\eps\to 0}\int_\G\int_\G
\frac{|f(x\cdot h)-f(x)-\langle v(x),\pi_x(h)\rangle_x|^p}{N(h)^p}
\rho_\eps(h)\,dh\,dx=0
\end{split}
\]
for some section $v\in L^p(\G,H\G)$. In this case, $v=\nabla_\G f$ almost everywhere.

If $p=1$, then in addition
\[
f\in BV_\G(\G)
\Longleftrightarrow
\limsup_{\eps\to 0}\int_\G\int_\G
\frac{|f(x\cdot h)-f(x)-\langle v(x),\pi_x(h)\rangle_x|}{N(h)}
\rho_\eps(h)\,dh\,dx<+\infty
\]
for some (and hence any) section $v\in L^1(\G,H\G)$.
\end{theorem}

The previous characterization was obtained in the Euclidean case $\G=\R^n$ by Spector \cite[Theorems 1.4 and 1.5]{Spector16} for mollifiers of the form $\rho_\eps(x)=\frac{1}{|B(0,\eps)|}\chi_{B(0,\eps)}(x)$ (see Example \ref{Examples}), and later extended to general mollifiers satisfying \eqref{P1}$\div$\eqref{P4} by Ponce and Spector \cite[Theorem 1.4 and Remark 1.5]{Ponce} (see also Brezis and Nguyen \cite[Proposition 1]{brezis-nguyen}).

The proof of Theorem \ref{secondcharBVSobfunct} proceeds as follows. We first establish in Theorem \ref{TaylorBV} the implication
\[
f\in BV_\G(\G)\Longrightarrow
\limsup_{\varepsilon\to 0}
\int_\G\int_\G
\frac{|f(x\cdot h)-f(x)-\langle v(x),\pi_x(h)\rangle_x|}{N(h)}
\rho_\varepsilon(h)\,dh\,dx
\le C\,|D_\G^{\rm s}f|(\G),
\]
with $v=\nabla^{\rm ac}_\G f$, following the approach of \cite[Theorem 1.4]{Ponce}. Here $\nabla^{\rm ac}_\G f$ and $D_\G^{\rm s}f$ denote, respectively, the density of the absolutely continuous part and the singular part in the decomposition of the horizontal derivative measure $D_\G f$ with respect to the $n$-dimensional Lebesgue measure (see \eqref{decompDGf}).

Moreover, this implication can be refined for functions $f\in W^{1,p}_\G(\G)$ (Theorem \ref{Taylormollgen}) by taking $v=\nabla_\G f$, showing that $f$ admits an $L^p$-Taylor approximation.

Conversely, if $f\in L^p(\G)$ with $1<p<\infty$ admits an $L^p$-Taylor approximation, then $f\in W^{1,p}_\G(\G)$; this is proved in Theorem \ref{converse}, which relies on Theorem \ref{firstcharacterization}. The more delicate case $p=1$ is addressed in Theorem \ref{converse2} and exploits the nonlocal approximation of the horizontal gradient. The proof of Theorem \ref{converse2} also yields the corresponding result for functions in $BV_\G(\G)$ (see Theorem \ref{converse2bv}).\\\\
We conclude by briefly outlining the structure of the paper.

In Section 2 we collect preliminary notions and results that will be needed throughout the paper. In particular, we recall some algebraic and metric properties of Carnot groups, the intrinsic differential calculus between Carnot groups, and the intrinsic notion of convolution between functions and measures in this setting, together with some of its properties.

Section 3 is devoted to the study of the nonlocal gradient $V_\eps(f)$ and of the associated functionals $\widetilde V_\eps(f)$, $I_{\eps,p}(f)$, and $I^*_{\eps,p}(f)$, which will lead to the proof of Theorem \ref{firstcharacterization}.

In Section 4 we recall Calder\'on--Zygmund differentiability in Carnot groups and related results. We prove Theorem \ref{secondcharBVSobfunct} on $L^p$-Taylor approximations and discuss some consequences.


\medskip

{\bf Acknowledgments.} The authors are members of the {\it Gruppo Nazionale per l’Analisi Matematica, la Probabilità e le loro Applicazioni} (GNAMPA), of the {\it Istituto Nazionale di Alta Matematica} (INdAM), and they are partially funded by the European Union under NextGenerationEU. PRIN 2022 Prot. n. 2022F4F2LH, the INdAM-GNAMPA 2025 Project \emph{Structure of sub-Riemannian hypersurfaces in Heisenberg groups}, CUP ES324001950001 and the INdAM-GNAMPA 2026 Project \emph{Variational, Geometric, and Analytic Perspectives on Regularity}, CUP E53C25002010001. The authors would like to thank V. Magnani for several useful suggestions concerning section 4 and   A. Pinamonti for very helpeful discussions on the topics of the present paper.

\section{Preliminaries}
\subsection{Carnot groups} In this section we recall basic and general facts on the structure of Carnot groups. General sources are for instance \cite{lanconelli}, \cite{folland}, \cite{heinonen95}, \cite{Pansu}, \cite{SerraCassano} and the references therein\vspace{-0.25 cm}.\linebreak\\
A \textit{Carnot group} $\G$ \textit{of step} $k$ is a simply connected Lie group equipped with a step $k$ \textit{stratification} of its Lie algebra $\mathfrak g$, namely there exist linear subspaces $V_1,\dots,V_k$ of $\mathfrak g$ such that
\[\mathfrak g=V_1\oplus\cdots\oplus V_k,\quad [V_1,V_i]=V_{i+1}\quad\text{for all }i=1,\dots,k-1,\quad V_k\neq \{0\}, \quad [V_1,V_k]=\{0\}.\]
We let $m_i:=\dim(V_i)$, $h_i=\sum_{j=1}^i m_j$ for $i=1,\dots,k$ and we fix an \textit{adapted basis} of $\mathfrak g$, i.e. a basis $X_1,\dots, X_n$ with the property that
$X_{h_{i-1}+1},\dots ,X_{h_i}$ is a basis of $V_i$ for every $i=1,\dots k$.
We will refer to $X_1,\dots, X_{m_1}$ as \textit{generating vector fields} of the group, while the subspace $V_1=\spann\{{X_1,\dots, X_{m_1}}\}$ is called the \textit{horizontal layer}.\medskip\\
The \textit{exponential map} is a global diffeomorphism
from $\mathfrak g$ onto $\G$: in particular, we can write any point $x\in\G$ in a unique way as $x = \exp(x_1X_1+\dots+ x_nX_n)$.
Using these exponential coordinates, we identify $x$ with the $n$-tuple $(x_1,\dots ,x_n)\in\R^n$ and we
can set $\G\equiv (\R^n,\cdot)$, where the explicit expression of the group operation $\cdot$ is given by the Campbell-Hausdorff formula (\cite{folland75}). We only recall here that $0\in\R^n$ corresponds to the identity element $e\in\G$ and that the inverse of $x$ is simply $-x$; some elementary properties of the group operation and the structure of vector fields in exponential coordinates are collected in \cite{FSSC03}.\medskip\\
The subbundle of the tangent bundle spanned by the generating vector fields of the group is called \textit{horizontal bundle} (or \textit{horizontal distribution}) and will be denoted by $H\G$. Explicitly, its fibers are given by 
\[H\G_x=\spann\{X_1(x),\dots,X_{m_1}(x)\}.\vspace{0.1 cm}\]
We define a \textit{sub-Riemannian structure} on $\G$, endowing each fiber of $H\G$ with a scalar product $\langle\cdot,\cdot\rangle_x$ which makes the basis $X_1(x),\dots ,X_{m_1}(x)$ orthonormal. We denote by $|\cdot|_x$ the norm induced on $H\G_x$.
\medskip\\
A vector $v\in H\G_x$ is called \textit{horizontal}, while the sections of $H\G$ are called \textit{horizontal vector fields}. Hence, horizontal vector fields can be expressed, point by point, as a linear combination of $X_1(x),\dots, X_{m_1}(x)$. This means that a horizontal section $X$ can always be identified with a function $\phi=(\phi_1,\dots,\phi_{m_1}):
\G\to \R^{m_1}$, simply by setting $ X(x)=\sum_{i=1}^{m_1}
\phi_i(x)X_i(x).$ In the sequel, we shall often use this identification, through which we can also write $|\phi(x)|=|\phi(x)|_x$.\medskip\\
For any $\lambda>0$ we define the \textit{intrinsic dilation} of a factor $\lambda$ as $\delta_ \lambda:\G\equiv \R^n\to\G\equiv \R^n$,
\[\delta_\lambda(x_1,\dots,x_n):=(\lambda^{\alpha_1}x_1,\dots,\lambda^{\alpha_n}x_n),\]
where $\alpha_i=j$ if $h_{j-1}+1\leq i\leq h_j$. Intrinsic dilations are automorphisms of the group $\G$.\medskip\\
A curve $\gamma:[0,T]\to\G$ is called \textit{sub-unit} (or
\textit{admissible}) if it is absolutely continuous (as a map from $[0,T]$ to $\R^n$) and if there exist real measurable functions $c_1(s),\dots,c_{m_1}(s)$, $s\in [0,T]$, such that $\sum_{i=1}^{m_1}c_j^2(s)\leq  1$ for a.e $s\in[0,T]$ and
\[\gamma'(s)=\sum_{j=1}^{m_1} c_j(s)X_j(\gamma(s))\quad \text{ for a.e. } s\in[0,T]. \]
If this is the case, \vspace{-0.1 cm}the \textit{speed} of the curve $\gamma$ at time $s$ is  $|\gamma'(s)|_{\gamma(s)}=\left(\sum_{j=1}^{m_1}c_j^2(s)\right)^{1/2}$. If we drop the requirement $\sum_{i=1}^{m_1}c_j^2(s)\leq  1$, then $\gamma$ is called \textit{horizontal} and its speed can be larger than 1. 
\medskip\\
If $x,y\in \G$, we define their \textit{Carnot-Carathéodory (CC) distance} $d_c(x,y)$ as
\[d_c(x,y)=\inf\{T>0: \exists \text{ a sub-unit curve $\gamma:[0,T]\to\G$ with $\gamma(0)=x$ and $\gamma(T)=y$}\}.\]
By Chow's theorem (see \cite{lanconelli}), the set of subunit curves joining $x$ and $y$ is non-empty. As a consequence, $d_c$ is a distance on $\G$ which induces the \textit{Euclidean topology}. In addition, $d_c$ is left-invariant and 1-homogeneous, namely
\[d_c(z\cdot x, z\cdot y)=d_c(x,y)\quad \text{and}\quad d_c(\delta_\lambda(x),\delta_\lambda(y))=\lambda\, d_c(x,y)\quad \forall\, x,y,z\in\G,\,\lambda>0.\vspace{0.1 cm}\] 
Any two points $x,y\in\G$ can actually be connected by a \textit{geodesic} $\gamma:[0,1]\to\G$, which is a horizontal curve having total variation (in the metric sense) equal to $d_c(x,y)$ (see \cite[Theorem 1.4.4]{monti}).\medskip\\
In general, if $d_1$ and $d_2$ are two \textit{invariant} distances on $\G$ (i.e. left-invariant and 1-homogeneous), then they are bilipschitz equivalent, i.e. there exist positive constants $C_1,C_2$ such that
\begin{equation}\label{equivalencenorms}C_1d_2(x,y)\leq d_1(x,y)\leq C_2d_2(x,y)\quad\text{for any $x,y\in\G$}.\end{equation}
The Euclidean distance $d_E$ is not invariant (unless $\G$ is a Euclidean space). Neverthless, for every compact set $A$ there exists a positive constant $C_A$ such that 
\[C_A^{-1}d_E(x,y)\leq  d_c(x,y)\leq C_A d_E(x,y)^{1/k}\quad\text{for all }x,y\in A.\]
Every invariant distance $d$ induces a \textit{homogeneous norm} $N_d=\|\cdot\|_d$ on $\G$, simply setting $\|x\|_d:=d(x,e)$, where $e$ is the identity of $\G$. Conversely, if $N$ is a homogeneous norm on $\G$ (namely $N(x)=0$ iff $x=e$, $N(x)=N(x^{-1})$, $N(\delta_\lambda(x))=\lambda N(x)$ and $N(x\cdot y)\leq N(x)+N(y)$), then it induces an invariant distance $d_N$ by setting $d_N(x,y):=N(x^{-1}y)$ for every $x,y\in\G$. \medskip\\ A map $\varphi:\G\to\G$ is called a \textit{horizontal rotation} if there exists a matrix $A\in O(m_1)$ such that \[\varphi(x_1,\dots,x_n)=(A(x_1,\dots,x_{m_1}),x_{m_1+1},\dots, x_n)\,,\]
where  $O(m_1)$ denotes the orthogonal group of matrices in  $\R^{m_1}$.
A homogeneous norm $N$ is called \textit{invariant under horizontal rotations} if $N(\varphi(x))=N(x)$ for any horizontal rotation $\varphi$ and $x\in\G$. Notice that, if $\G$ is a Carnot group, there exists always a homogeneous norm $N$ which is invariant under horizontal rotations (see \cite[Theorem 5.1]{FSSC03}).
\medskip\\ In the following, we fix a homogeneous norm $N$ on $\G$ (or equivalently, an invariant distance $d$).\linebreak We will denote by $B(x,r)=B_N(x,r)$ the open ball computed with respect to $N$. When considering the Carnot-Carathéodory distance $d_c$, we will denote by $\|\cdot\|_c$ and by $B_c(x,r)$ the corresponding homogeneous norm and the associated balls. Moreover, we will also use the notation $B(r):=B(e,r)$ and $B_c(r):=B_c(e,r)$.\medskip\\
The integer $Q$ defined as
\[Q:=\sum_{i=1}^k i\dim(V_i)\]
is called the \textit{homogeneous dimension} of $\G$ and plays a fundamental role in this theory. The measure $\mathcal L^n$ on $\R^n\equiv \G$ is the \textit{Haar measure} of the group, it is both left- and right-invariant and is $Q$-homogeneous with respect to intrinsic dilations:
\[\mathcal L^n(x\cdot E)=\mathcal L^n(E)=\mathcal L^n(E\cdot x),\quad\mathcal L^n(\delta_\lambda(E))=\lambda^Q\mathcal L^n(E)\quad\text{ for any }x\in\G,\lambda>0.\]
As a consequence, if $d$ is any invariant distance on $\G$, then the space $(\G,d,\mathcal L^n)$ is \textit{Ahlfors $Q$-regular} and the metric dimension of $\G$ is $Q$.
More precisely, it holds that \[\mathcal L^n(B_N(x,r))=C_N\,r^Q\]for any $x\in\G, r>0$ and any homogeneous norm $N$ on $\G$. We will also denote by $|E|$ the Lebesgue measure of a set $E\subset\G$.
\medskip


\subsection{Differential calculus on Carnot groups}\phantom{}\vspace{0.1 cm}\\
In this section, we denote by $\G_1$ and $\G_2$ two Carnot groups, equipped with associated intrinsic dilations $\delta_\lambda^1$ and $\delta_\lambda^2$. We also fix two homogeneous norms $\|\cdot\|_1$, $\|\cdot \|_2$ on $\G_1$ and $\G_2$, respectively.

\begin{definition}[H-linear maps]\label{Pdiff}
  We say that 
  $L:\G_1\to\G_2$ is
   \emph{H-linear}, or
  it is a \emph {homogeneous homomorphism}, if $L$ is a group
   homomorphism  such that 
$$
   L(\delta^1_\lambda x)=\delta^2_\lambda L(x), \quad \text{for all  }x\in \G_1 \text{ and } \lambda>0.
$$
\end{definition}
In the real-valued case $\G_2= \R$, there exists an easy characterization of H-linear maps. For $x\in\G_1$, we denote by $\pi_x:\G_1\to H\G_1$ the smooth section defined by 
\[\pi_x(y):=\sum_{i=1}^{m_1}y_iX_i(x).\]
\begin{proposition}\cite[Proposition 2.5]{FSSC03}\label{Hlinear}
    A map $L:\G\to \R$ is H-linear if and only if there is $a=(a_1,\dots,a_{m_1})\in\R^{m_1}$ such that \[L(y)=\sum_{i=1}^{m_1}a_iy_i.\] Equivalently, for any $x\in\G$ we can write 
$L(y)=\langle a,\pi_x(y)\rangle_x$, where we are using the identification $\R^{m_1}\equiv H\G_x$ by setting $a=\sum_{i=1}^{m_1}a_iX_i(x)\in H\G_x$.
\end{proposition}
The fundamental notion of pointwise differentiability for functions acting between Carnot groups is due to Pansu \cite{Pansu}.
\begin{definition}[Pansu differentiability]\label{pansudifferentiability} Let $\Omega\subset\G_1$ be open.
We say that $f:\Omega\to\G_2$ is \emph {P-differentiable} at $x\in\Omega$ if
  there is a $H$-linear function $L:\G_1\rightarrow\G_2$ such that
\begin{equation*}
  \|\left(L(x^{-1}\cdot
  y)\right)^{-1}\cdot f(x)^{-1}\cdot f(y)\|_2= o\big( \| x^{-1}\cdot y\|_1\big),\quad  \text {as $\|x^{-1}\cdot y\|_1\to 0$}.
\end{equation*}
The $H$-linear map $L$, also denoted by $D_Pf(x)$, is called  the \emph{Pansu's differential} of $f$ at $x$.
\end{definition}
Notice that, by \eqref{equivalencenorms}, the property of being P-differentiable at a point $x\in\Omega$ does not depend on the chosen homogeneous norms on $\G_1$ and $\G_2$. If $\G_2\equiv\R$, we will naturally consider $\|\cdot\|_2=|\cdot|$.
In this latter case, in view of Proposition \ref{Hlinear}, $f:\Omega\to\R$ is P-differentiable at $x\in\Omega$ if there exists $a_x\in\R^{m_1}$ 
such that
\[\lim_{y\to x}\frac{|f(y)-f(x)- \langle a_x,\pi_x(x^{-1}y)\rangle_x|}{d_1(x,y)}=0\]
and in particular it holds $D_Pf(x)(z)=\langle a_x,\pi_x(z)\rangle_x$.\medskip\\
If $E\subset\G_1$, we denote by Lip$\,(E,\G_2)$ the space of all Lipschitz continuous maps $f:E\to\G_2$. Since any two invariant distances on a Carnot group are bilipschitz equivalent, this space does not depend on the fixed homogeneous norms on $\G_1$ and $\G_2$. If $\G_2=\R$, we usually consider $\G_1=\G$ and we use the notation $f\in\text{Lip}_\G(E)$. Finally, if $f$ has also compact support in an open set $\Omega$, we write $f\in\text{Lip}_{\G,c}(\Omega)$.  The fundamental result where P-differentiability applies is the following generalization of the Rademacher theorem.
\begin{theorem}[Pansu-Rademacher Theorem \cite{Pansu}]\label{pansuthm}
    Let $f\in \emph{Lip}(\Omega,\G_2)$, where $\Omega\subset\G_1$ is an open set. Then $f$ is P-differentiable $\mathcal L^n$-a.e. in $\Omega$. 
\end{theorem}
We denote by $C^k(\Omega)$ the set of all real-valued functions of class $C^k$ defined on $\Omega\subset\G\equiv\R^n$, while we use the notation $C^k(\Omega,H\G)$ to denote the set of all sections of $H\G\equiv \R^{m_1}$ of class $C^k$.\linebreak Clearly, we also write $C^k_c(\Omega)$ (resp. $C^k_c(\Omega, H\G)$) for functions (resp. sections) with compact support in $\Omega$, while $C^k_b(\Omega)$ (resp. $C^k_b(\Omega, H\G)$) stands for bounded functions (resp. sections).\medskip\\
If $f\in C^1(\Omega)$, we denote by $\nabla_\G f\in C^0(\Omega,H\G)$ the \textit{horizontal gradient} of $f$, defined as the section 
\[\nabla_\G f:=\sum_{i=1}^{m_1}(X_if)X_i\]  or simply, in coordinates, $\nabla_\G f =(X_1f,\dots,X_{m_1}f)$.
Dually, for $\varphi=(\varphi_1,\dots,\varphi_{m_1})\in C^1(\Omega,H\G)$, we will instead denote by div$_\G\varphi\in C^0(\Omega)$ the \textit{horizontal divergence} of $\varphi$, defined by
\[\text{div}_\G\varphi:=\sum_{i=1}^{m_1}X_i\varphi_i.\]
\begin{example}[The Heisenberg group]\label{H1}
The (first) Heisenberg group $\G=\hh^1=(\R^3,\cdot)$ is the (simplest) non-abelian Carnot group endowed with the following group law: if $x=(x_1,x_2,x_3)$ and $y=(y_1,y_2,y_3)$,
\[
x\cdot y:=\,\left(x_1+y_1,x_2+y_2,x_3+y_3+\frac{1}{2} (x_1y_2-x_2y_1)\right)\,.
\]
It is equipped with the family of dilations $\delta_\lambda:\,\hh^1\to \hh^1$, defined as
\[
\delta_\lambda(x):=\,\left(\lambda\,x_1,\lambda\,x_2, \lambda^2\,x_3\right)\text{ if }x=(x_1,x_2,x_3),\,\lambda>\,0\,.
\]
If $\mathfrak h^1$ denotes the Lie algebra associated to $\hh^1$, a basis of $\mathfrak h^1$ is given by the left-invariant vector fields
\[
X_1(x):=\,\partial_1-\frac{x_2}{2}\partial_3,\,X_2(x):=\,\partial_2+\frac{x_1}{2}\partial_3,\,X_3(x):=\,\partial_3\,
\]
and the only nontrivial commutation between them is
\[
[X_1,X_2]=\,X_3\,.
\]
Then a stratification of $\mathfrak h^1$  is
\[
\mathfrak h^1=V_1\oplus V_2
\]
with 
\[
V_1:=\,{\rm span}\{X_1,X_2\}\text{ and }V_2:={\rm span}\{X_3\}.
\]
Thus $\hh^1$ turns out to be a Carnot group of step $k=2$, and topological and homogeneous dimension $n=3$ and $Q=4$, respectively. The horizontal gradient $\nabla_\G=\,\nabla_{\hh^1}:=\,(X_1,X_2)$ (see, for instance, \cite[Chapter 10]{stein}).
A very relevant homogeneous norm on $\hh^1$ is the so called {\it Cygan-Korányi} norm defined as
\[
N(x):=\,\left((x_1^2+x_2^2)^2+16 x_3^2\right)^{1/4}\text{ if } x=(x_1,x_2,x_3)\in\hh^1\,
\]
(see, for instance, \cite[Proposition 14.2.5]{LeDonne}). It is immediate the $N\in C^\infty(\R^3\setminus\{0\})$ and it is invariant w.r.t. the horizontal rotations. A simple calculation gives that
\[
\begin{split}
\nabla_{\hh^1}N(x)&=\,\left(X_1 N(x), X_2 N(x)\right)\\
&=\,\frac{1}{N(x)^3}\left((x_1^2+x_2^2)x_1-4 x_2 x_3, (x_1^2+x_2^2)x_2+4 x_1 x_3\right)\text{ for each }x\in\hh^1\setminus\{0\}\,,
\end{split}
\]
and 
\[
|\nabla_{\hh^1}N(x)|=\,\frac{\sqrt{x_1^2+x_2^2}}{N(x)}\text{ for each }x\in\hh^1\setminus\{0\}\,.
\]
In particular, note that 
\[
|\nabla_{\hh^1}N(x)|=\,0\text{ for each }x\in\{(0,0,x_3):\,x_3\neq\,0\}\,.
\]
This is a typical feature of sub-Riemannian geometry: unlike in the Euclidean case, the (horizontal) gradient of a norm may vanish away from the origin.
\end{example}
The space $L^p(\Omega,H\G)$ denotes the set of all sections $\varphi=(\varphi_1,\dots,\varphi_{m_1})$ which are measurable and such that
\[\|\varphi\|^p_{L^p(\Omega)}:=\left\{\int_\Omega|\varphi(x)|_x^p\,dx\right\}<+\infty.\]
Sobolev spaces for scalar functions defined on a Carnot group $\G$ are constructed in
the usual way by considering distributional derivatives along the basis $X_1,\dots,X_{m_1}$ of the horizontal layer $V_1$ of $\G$: a function $g\in L^1_{\rm loc}(\Omega)$ is the \textit{weak} (or \textit{distributional}) \textit{derivative} of $f\in L^1_{\rm loc}(\Omega)$ along the vector field $X_i$ if
\[\int_\Omega g\varphi=-\int_\Omega f X_i \varphi\quad\text{ for all }\varphi\in C^1_c(\Omega).\]
As usual, we will write $X_i f$ to denote the weak derivative of $f$ along $X_i$ and \[\nabla_\G f := (X_1f,\dots,X_{m_1} f)\] for the \textit{weak horizontal gradient}. If $f:\Omega\to\R$ is continuous and the same holds for its distributional derivatives $X_if$, $i=1,\dots,m_1$, then we say that $f\in C^1_\G(\Omega)$. In this case, $X_if$ exists also in classical sense, by computing the derivative of $f$ along the vector field $X_i$. Notice that $C^1(\Omega)\subset C^1_\G(\Omega)$ and the inclusion is actually strict (see, for instance, \cite[Remark 5.9]{FSSC01}).
\begin{definition}[Horizontal Sobolev space]\label{horSobspace} For $1\leq p\leq +\infty$, we define the (\textit{horizontal}) \textit{Sobolev space} $W^{1,p}_\G(\Omega)$ as the space of all functions $f\in L^p(\Omega)$ having weak horizontal gradient $\nabla_\G f\in L^p(\Omega, H\G)$, equipped with the norm
\[\|f\|_{W^{1,p}_\G(\Omega)}:=\|f\|_{L^p(\Omega)}+\|\nabla_\G f\|_{L^p(\Omega)}.\]
As usual, we write $W^{1,p}_{\G, \rm loc}(\Omega)$ to denote the space of functions in $W^{1,p}_{\G}(U)$ for every $U\Subset\Omega$. 
\end{definition} 

Meyers-Serrin type approximation theorems for function spaces associated with a family of (locally) Lipschitz vector fields are well known and in particular they apply to the case of Carnot-type vector fields (see \cite[Theorem 1.2.3]{FSSC96} or \cite[Theorem 1.13]{garofalo}).
\begin{theorem}\label{meyersserrin}
 Let $1\leq p<+\infty$. The space $C^\infty(\Omega)\cap W^{1,p}_\G(\Omega)$ is dense in $W^{1,p}_\G(\Omega)$.
 
\end{theorem}
If $f\in \text{Lip}_{\G,c}(\Omega)$, then $f\in W^{1,\infty}_\G(\Omega)$ (see \cite[Theorem 1.3]{GN98}). In particular, \[X_i f(x)=D_Pf(x)(e_i)\]
for almost every $x\in\Omega$ (cf. \cite[Remark 3.3]{montiserra}); hence weak and classical derivatives coincide.
\begin{definition}[Functions of bounded variation]\label{bvG}
    The set $BV_\G(\Omega)$ of \textit{functions of bounded horizontal variation} is the set of all maps $f\in L^1(\Omega)$ such that
    \[|D_\G f|(\Omega):=\sup\left\{\int_\Omega f\,\text{\rm div}_\G\varphi : \varphi\in C^1_c(\Omega,H\G), |\varphi(x)|_x\leq 1\right\}<+\infty.\]
    We write $BV_{\G, \rm loc}(\Omega)$ to denote the space of functions which are in $BV_\G(U)$ for every $U\Subset\Omega$.
\end{definition}

\begin{remark} It is easy to see that both the (Euclidean) Sobolev space 
$W^{1,p}_{\mathrm{loc}}(\Omega)$ (with $1 \le p \le \infty$) 
and the space $BV_{\mathrm{loc}}(\Omega)$ 
are (strictly) contained in the horizontal Sobolev space 
$W^{1,p}_{\mathbb{G}, \mathrm{loc}}(\Omega)$ 
and in the space $BV_{\mathbb{G}, \mathrm{loc}}(\Omega)$, respectively.
\end{remark}

\begin{theorem}[Structure theorem for $BV_\G(\Omega)$]\label{structurethm}
  Let $f\in BV_{\G, \rm loc}(\Omega)$. Then $|D_\G f|$ is a Radon measure on $\Omega$ (which is finite if $f\in BV_\G(\Omega))$. Moreover, there exists a $|D_\G f|$-measurable horizontal section $\sigma_f:\Omega\to H\G$ with $|\sigma_f(x)|_x=1$ for $|D_\G f|$-a.e. $x\in\Omega$ such that
\[\int_\Omega f\,\text{\rm div}_\G\varphi=-\int_\Omega \langle\varphi, \sigma_f\rangle_x\,d|D_\G f|\quad\text{ for all }\varphi\in C^1_c(\Omega, H\G).\]
\end{theorem}
\noindent The measure $|D_\G f|$ is called the \textit{total variation measure} of $f$. We finally let $D_\G f:=\sigma_f\cdot |D_\G f|$ to be the \textit{distributional horizontal gradient} of $f\in BV_\G(\Omega)$, that is identified with a $\R^{m_1}$-valued Radon measure on $\Omega$, whose total mass actually coincides with $|D_\G f|.$ \medskip\\
Recall that, if $\mu$ is a $\R^m$-valued Radon measure on $\Omega$, the \textit{total mass} (or \textit{variation}) of $\mu$ is the positive measure defined by
\[|\mu|(E):=\sup\left\{\sum_{h=1}^\infty|\mu(E_h)|:E_h \text{ Borel sets},\,E\subset\bigcup_{h=1}^\infty E_h\right\}.\]
We decompose 
\begin{equation}\label{decompDGf}
D_\G f=\nabla^{\rm ac}_\G f\,\mathcal L^n+ D^s_\G f,
\end{equation}
where $\nabla^{\rm ac}_\G f$ is the density of the \textit{absolute continuous part} of $D_\G f$ with respect to $\mathcal L^n$, while $D^s_\G f$ denotes the \textit{singular part} of the measure. 
Notice that $W^{1,1}_\G(\Omega)\subset BV_\G(\Omega)$ and in this case $D_\G f=\nabla_\G^{\rm ac}f\,\mathcal L^n=\nabla_\G f\, \mathcal L^n$. In particular  \[|D_\G f|(\Omega)=\int_\Omega|\nabla_\G f|\,dx\quad\text{for all } f\in W^{1,1}_\G(\Omega).\]
We equip the space $BV_\G(\Omega)$ with the norm
\[\|f\|_{BV_\G(\Omega)}:=\|f\|_{L^1(\Omega)}+|D_\G f|(\Omega).\]
As expected, smooth functions are not dense in $BV_\G(\Omega)$. Neverthless, as in the Euclidean case, an Anzellotti-Giaquinta type theorem holds for spaces associated to Lipschitz vector fields. In our context, the result reads as follows (see \cite[Theorem 1.14]{garofalo} or \cite[Theorem 2.2.2]{FSSC96}).
\begin{theorem}\label{anzellotti}
    Let $f\in BV_\G(\Omega)$. There exists a sequence $(f_h)_h\subset C^\infty(\Omega)\cap BV_\G(\Omega)$ such that
    \[f_h\to f\;\;\text{in }L^1(\Omega);\qquad|D_\G f_h|(\Omega)\to |D_\G f|(\Omega).\]
\end{theorem}
\begin{definition}
    Given a measurable set $E\subset \G$, we say that $E$ has \textit{finite $\G$-perimeter} if $\chi_E$ is a function of bounded horizontal variation. In this case, according to Theorem \ref{structurethm},  $|D_\G\chi_E|$ is a finite Radon measure, which is called the $\G$-\textit{perimeter measure} of $E$ and is denoted by $|\partial E|_\G$.
\end{definition}
We recall the following classical result of integration in polar coordinates, which combines \cite[Proposition 1.15]{folland} and \cite[Corollary 4.6]{montiserra}:
\begin{theorem}[Integration in polar coordinates]\label{polarcoordinates}
   Let $\G$ be a Carnot group endowed with a homogeneous norm $N$. Let \[S=S_\G:=\{x\in\G: N(x)=1\}.\] There is a unique Radon measure $\sigma$ on $S$ such that, for all $u\in L^1(\G)$, 
   \begin{equation}\label{polarstein}\int_\G u(x)\,dx=\int_0^{+\infty}\int_S u(\delta_r(y))\,r^{Q-1}\,d\sigma(y)\,dr.\end{equation}
 Moreover, if $N=\|\cdot\|_{c}$, then it holds that $\sigma=|\partial B_c(1)|_\G$.
 \end{theorem}
 Notice that if $u$ is a radial function, namely $u(x)=\widetilde u(N(x))$, then \eqref{polarstein} takes the form
 \begin{equation}\label{iwrtradfctpoco}
 \int_\G u(x)\,dx=\sigma(S_\G)\int_0^{+\infty}\widetilde u(r)\, r^{Q-1}\,dr
 \end{equation}


We close this section by recalling the classical Lebesgue differentiation theorem on Carnot groups, which follows from the doubling property of the space (see \cite[Proposition 1.10]{magnani}).
\begin{proposition}\label{lebesguepoint}
    Let $u\in L^p(\Omega,\R^m)$ for $1\leq p<+\infty$. For almost every $x\in\Omega$ it holds
    \begin{equation}\label{eq_lebesguepoint}\lim_{r\to 0}\fint_{B(x,r)}|u(y)-u(x)|^p\,dy=0.\end{equation}
\end{proposition}
A point $x\in\Omega$ satisfying \eqref{eq_lebesguepoint} is called a \textit{Lebesgue point} of $u\in L^p(\Omega,\R^m)$. In particular, Lebesgue points are points of \textit{approximate continuity} for the function $u$.

 \subsection{Convolution on Carnot groups}
 \begin{definition}[Approximation to the identity]\label{def_approximation}
     An \textit{approximation to the identity} in $\G$ is a family of functions $(K_\eps)_\eps\subset L^1(\G)$, $\eps>0$, such that
     \begin{enumerate}
         \item [(i)] The family $(K_\eps)_\eps$ is bounded in $L^1(\G)$, that is
         \[\sup_{\eps>0} \|K_\eps\|_{L^1(\G)}<+\infty;\]
         \item[(ii)]For every $\eps>0$ it holds\[\int_\G K_\eps(x)\,dx=1;\]
         \item [(iii)] For any $\delta>0$
         \[\lim_{\eps\to 0}\int_{\G\setminus B(\delta)}|K_\eps(x)|\,dx=0.\]
     \end{enumerate}
 \end{definition}
 \begin{example}\label{example}
     The simplest and most frequent example of an approximation to the identity in $\G$ can be constructed by considering the family
     \[K_\eps(x):=\frac{1}{\eps^Q}\phi(\delta_\frac{1}{\eps}(x)),\]
    where $\phi\in L^1(\G)$ is such that $\int_\G \phi(x)\,dx=1$.  \medskip\\
    If in addition $\phi\in C^{\infty}_c(\G)$ is such that $0\leq \phi\leq 1$, $\phi(x^{-1})=\phi(x)$ and ${\rm spt}\,\phi\subset B(0,1)$, then the corresponding $(K_\eps)_\eps$ form a family of intrinsic mollifiers on $\G$. 
 \end{example}

 \begin{definition}[Convolution between functions]
     If $f$ and $g$ are Borel functions on $\G$, their \textit{ intrinsic convolution} is defined by 
     \[(f\ast g)(x):=\int_\G f(x\cdot y^{-1})\,g(y)\,dy=\int_\G f(z)\,g(z^{-1}\cdot x)\,dz\]
     provided that the integrals above converge. If $f$ (or $g$) is a vector-valued map, then the intrinsic convolution is simply defined componentwise.
     \end{definition}
     
     \begin{proposition}\label{young}{\rm {(Young convolution inequality, see \cite[Proposition 1.18]{folland})}}\\
         Assume $1\leq p,q,r\leq+\infty$ and $p^{-1}+q^{-1}=r^{-1}+1$. If $f\in L^p(\G)$ and $g\in L^q(\G)$, then $f\ast g\in L^r(\G)$ and 
         \[\|f\ast g\|_{L^r(\G)}\leq \|f\|_{L^p(\G)}\|g\|_{L^q(\G)}.\]
     \end{proposition}
     \begin{lemma}[Properties of convolution]\label{propconv1}
         Let $(K_\eps)_\eps$ be an approximation to the identity on $\G$ and let also $\eta\in C^\infty_c(\G)$. 
         \begin{enumerate}
             \item [(i)] If $f\in L^p(\G)$ for $1\leq p<+\infty$, then
             \[\|K_\eps\ast f-f\|_{L^p(\G)}\to 0\quad\text{and}\quad \|f\ast K_\eps-f\|_{L^p(\G)}\to0;\]
             \item [(ii)] If $f\in C^0_b(\G)$, then $K_\eps\ast f\to f$  (and $f\ast K_\eps\to f)$ uniformly on compact sets;
             \item[(iii)] If $f\in W^{1,p}_\G(\G)$, then $\eta\ast f\in C^{\infty}(\G)$ and 
             \[\nabla_\G({\eta}\ast f)={\eta}\ast \nabla_\G f;\]
             \item[(iv)] If $f\in L^p(\G)$ and $K\in L^1(\G)$,
             \[({\eta}\ast f)\ast K={\eta}\ast (f \ast K).\]
         \end{enumerate}
     \end{lemma}
     \begin{proof}
     The proof of point (i) can be found in \cite[Proposition 1.20, (i)]{folland} and \cite[Proposition 1.28, (i)]{vittone} in the case $K_\eps$ has the form of Example \ref{example}; however a similar argument also works for a general approximation to the identity. The same holds for property (ii) (\cite[Proposition 1.20, (iii)]{folland} and \cite[Proposition 1.28, (v)]{vittone}). Concerning point (iii), see \cite{folland}, \cite[Proposition 1.28]{vittone} or \cite[Proposition 2.14]{comi} for the smooth case, while the Sobolev case can be obtained by a standard approximation argument.
     Item (iv) follows by an explicit computation. \end{proof}
     \begin{definition}[Convolution between vector measures and functions]
         Assume $\nu$ is a $\R^m$-valued Radon measure on $\G$. Let $f:\G\to\overline \R$ be a Borel function. Then the \textit{intrinsic convolution} between $\nu$ and $f$ (or between $f$ and $\nu$) is given by the (vector) function
         \[\begin{split}(\nu\ast f)(x):=\int_\G f(y^{-1}\cdot x)\, d\nu(y),\\ (f\ast \nu)(x):=\int_\G f(x\cdot y^{-1})\,d\nu(y),
         \end{split}\]
         assuming that the above integrals make sense. 
     \end{definition}We now recall the following definition of weak convergence for Radon measures.
\begin{definition}[Weak convergence of measures]
Let $(\mu_h)_{h\in\mathbb N}$ and $\mu$ be signed Radon measures on $\G$. We say that $\mu_h$  \textit{locally weakly$^*$ converges} to $\mu $, and we write $\mu_h\rightharpoonup\mu$, if for every $\varphi\in C^0_c(\G)$ it holds
\[\int_\G\varphi\,d\mu_h\longrightarrow \int_\G\varphi\,d\mu\quad\text{ as }h\to\infty.\]
We denote by $\mathcal M(\G)$ the space of signed Radon measures on $\G$, while we use the notation $\mathcal M(\G,\R^{m})$ for the space of $\,\R^{m}$-valued Radon measures.
Finally, if $\mu_h,\,\mu\in \mathcal M(\G,\R^{m})$, we write $\mu_h\rightharpoonup\mu$ if the convergence holds componentwise.
\end{definition}
     \begin{lemma}[Properties of convolution of measures]\label{propconv2} Let $\nu$ be a finite $\R^m$-valued Radon measure on $\G$. Let $(K_\eps)_\eps$ be an approximation to the identity on $\G$ and let also $\eta\in C^\infty_c(\G)$. 
         \begin{enumerate}
             \item[(i)]  If $f\in L^p(\G)$ for $1\leq p\leq +\infty$, then $\nu\ast f$ is well defined a.e., $\nu\ast f\in L^p(\G)$ and
             \[\|\nu\ast f\|_{L^p(\G)}\leq \|f\|_{L^p(\G)}\,|\nu|(\G).\]
             The same holds for $f\ast \nu$;
             \item[(ii)] If $f\in C^0_b(\G)$, then also $\nu\ast f\in C^0_b(\G)$ and $f\ast \nu\in C^0_b(\G)$;
             \item[(iii)] It holds
             \[\nu\ast K_\eps\rightharpoonup\nu\quad\text{and}\quad K_\eps\ast \nu\rightharpoonup \nu\,;\]
              Moreover, if $\|K_\varepsilon\|_{L^1(\G)}=\,1$ for each $\varepsilon>\,0$, it also holds
             \[
             |\nu|(\G)=\,\lim_{\eps\to 0}\|\nu\ast K_\eps\|_{L^1(\G)}=\,\lim_{\eps\to 0}\|K_\eps\ast\nu \|_{L^1(\G)}\,.
             \]
             
             \item[(iv)] If $f\in BV_\G(\G)$, then $\eta\ast f\in C^{\infty}(\G)$ and
             \[\nabla_\G(\eta\ast f)=\eta\ast D_\G f;\]
             \item [(v)] For $K\in L^1(\G)$, it holds \[(\eta\ast\nu)\ast K=\eta\ast (\nu \ast K).\]
         \end{enumerate}
     \end{lemma}
     \begin{proof} 
     The first item can be proved as in the Euclidean case, see \cite[Proposition 8.49]{folland99}. Point (ii) follows by a direct computation using Lebesgue's dominated convergence theorem. Property (iii) is contained in \cite[Remark 2.12]{comi} for the case of mollifiers, but the argument can be extended to general approximations to the identity. The convergence of the total variation follows from the lower semicontinuity of the total variation w.r.t. the weak convergence of measures (see, for instance, \cite[Theorem 1.59]{AFP}) and the previous property (i).
     Point (iv) is contained in \cite[Lemma 3.10]{comi}, while the last item follows by an explicit computation.
     \end{proof}

\section{Nonlocal approximations of the horizontal gradient}\label{section4}
\noindent 
In this section we extend to the setting of Carnot groups an interesting representation formula which was recently established by Brezis and Mironescu \cite{brezis} in the Euclidean case. As a consequence, we derive a nonlocal approximation theorem for the horizontal gradient, which is of independent interest and will be used in the next section. This asymptotic estimate falls within the BBM framework. We also derive a characterization of Sobolev and BV functions on $\G$ in terms of boundedness of non-local energies, which is independent of the chosen homogeneous norm and generalizes a result of Barbieri \cite{barbieri}.\medskip\\
From now on, we assume that $\G$ is a fixed Carnot group of homogeneous dimension $Q$, endowed with a homogeneous norm $N$. 
In the following, unless otherwise specified, we consider a family of functions $(\rho_\eps)_\eps$ on $\G$ satisfying 
$\eqref{P1}\div\eqref{P4}$.
    In particular, 
    $(\rho_\varepsilon)_\varepsilon$ is a positive and radial approximation to the identity (compare with Definition \ref{def_approximation}).
 \begin{remark} Observe that, for a fixed homogeneous norm $N$ on $\G$, properties \eqref{P1}$\div$ \eqref{P4} on the family of mollifiers $(\rho_\eps)_\eps$
can be equivalently read in terms of their profiles $(\tilde\rho_\eps)_\eps$ according to the integration formula in polar coordinates \eqref{iwrtradfctpoco}:
\begin{equation}\label{tildeP1}\tag{$\widetilde {P1}$}
\tilde\rho_\eps:\,[0,+\infty)\to [0,+\infty) \text{ is a Borel map};
\end{equation}
\begin{equation}\label{tildeP3}\tag{$\widetilde {P3}$}
\int_0^\infty\tilde\rho_\eps(r)\,r^{Q-1}\,dr=\,\sigma(S_\G)^{-1};
\end{equation}
\begin{equation}\label{tildeP4}\tag{$\widetilde {P4}$}
\lim_{\eps\to 0}\int_\delta^\infty\tilde\rho_\eps(r)\,r^{Q-1}\,dr=\,0\text{ for each } \delta>\,0.
\end{equation}
Notice also that assumption \eqref{P3} can be often weakened to the following:
\begin{equation}\label{P3'}\tag{P3'}
\exists \,C>0: \int_\G \rho_\eps(h)\,dh= C\quad\text{for every }\eps>0.\end{equation}
Indeed, assuming \eqref{P3'}, one can construct a family of mollifiers satisfying \eqref{P3} simply by normalization, while preserving all the other properties. In particular, Theorem \ref{firstcharacterization} and Theorem \ref{secondcharBVSobfunct} remain valid under assumption \eqref{P3'} replacing the stronger \eqref{P3}, with the only difference that in this case $V_\eps(f)\to C\,\nabla_\G f$ in $L^p(\G)$ when $f\in W^{1,p}_\G(\G)$, while $V_\eps(f)\to C\, D_\G f$ in $\mathcal M(\G,\mathbb R^{m_1})$ if $f\in BV_\G(\G).$
\end{remark}
We now provide some interesting examples of mollifiers $\rho_\eps$ arising from the existing literature:
\begin{example}\label{Examples}(Examples of mollifiers) \\
The prototype example of a family of mollifiers $(\rho_\eps)_\eps$ satisfying assumptions $\eqref{P1}\div\eqref{P4}$ is the following:
\[\rho_\eps(h):=\frac{\chi_{B(0,\eps)}(h)}{|B(0,\eps)|}.\]
Calderón-Zygmund differentiability and the original definition of $L^p$-Taylor approximation by Spector \cite{spector} are indeed modeled on a family $(\rho_\eps)_\eps$ of this form (see Section 4 and the references therein). 
Another significant example of a family of mollifiers satisfying assumptions $\eqref{P1}\div\eqref{P4}$ is given by 
\[\rho_\eps(h):=\frac{c_\eps \chi_{B(0,R)}(h)}{N(h)^{Q-\eps p}},\]
where $c_\eps:=\frac{\eps p}{R^{\eps p}\sigma(S_\G)}$ and $R,p>0$ are fixed. 
This class of kernels has been widely used in the literature, since it appears in the definition of the fractional Gagliardo seminorms and it was the initial motivation for the study of BBM-type estimates (see \cite{BBM}). \end{example}
We now deal with the definition and properties of the nonlocal gradients $V_\eps(f)$ and the related functionals presented in the introduction.
\begin{remark} 
   By the triangular inequality, if $N$ is a homogeneous norm, for every $x,y\in\G$
    \[|N(x)-N(y)|\leq N(y^{-1}x)=d_N(x,y).\]
    Hence $N:(\G,d_N)\to \R$ is 1-Lipschitz continuous. In particular, by Pansu-Rademacher theorem (Theorem \ref{pansuthm}), $N$ is P-differentiable a.e. on $\G$. Moreover, $\nabla_\G N$ exists in distributional sense and $\nabla_\G N\in L^\infty (\G,H\G)$ (see, for instance, \cite[Theorem 1.3] {GN98}.) 
\end{remark} \begin{remark}\label{NeuclidthenBrezis} If $N$ is the standard Euclidean norm on $\R^n$, then \[\nabla_\G N(h)=\nabla N(h)=\frac{h}{N(h)}\] for all points $h\neq 0$ and the nonlocal gradient $V_\eps(f)$ exactly reduces to the expression used in \cite{mengesha} and then in \cite{brezis}.\end{remark}
 
\begin{remark}\label{eikeqccnor} If $N=\|\cdot\|_c$ denotes the homogeneous norm induced by the CC distance on $\G$, then equation \eqref{eikonaleq} is satisfied a.e. (see \cite[Theorem 3.1]{montiserra}) and, in particular, $I_{\eps, p}(f)=I^\ast_{\eps,p}(f)$.
\end{remark}


  We now deal with the existence of $V_\eps(f).$ 
 It is immediate that the following estimate holds.
 \begin{lemma}\label{firstestimate}
     Let $f\in L^1_{\rm loc}(\G)$ be such that $\widetilde V_\eps(f)\in L^1_{\rm loc}(\G)$. Then $V_\eps(f)$ is well-defined a.e. and is measurable. Moreover we have
     \[|V_\eps (f)|\leq Q\,\widetilde V_\eps (f)\qquad\text{a.e. in }\G.\]
     In particular also $V_\eps(f)\in L^{1}_{\rm loc}(\G).$ 
     \end{lemma}
We now present a sufficient condition to ensure that $\widetilde V_\eps(f)\in L^1_{\rm loc}(\G)$ (and thus, by Lemma \ref{firstestimate}, also that $V_\eps (f)\in L^1_{\rm loc}(\G)$).

\medskip
\begin{lemma}\label{secondestimate}
   Let $1\leq p<+\infty$ and $f\in L^1_{\rm loc}(\G)$. Let also $\rho_\eps$ satisfy \eqref{P1} and \eqref{P2}. Then
   \[\|\widetilde V_\eps(f)\|^p_{L^p(\G)}\leq \|\rho_\eps\|_{L^1(\G)}^{p-1} I_{\eps,p}(f).\]
   Consequently, if $I_{\eps, p}(f)<+\infty$, then $V_\eps(f)(x)$ is well defined a.e. and 
   \[\|V_\eps(f)\|_{L^p(\G)}\leq Q\,\|\rho_\eps\|_{L^1(\G)}^{(p-1)/p}\, I_{\eps,p}(f)^{1/p}\]
\end{lemma}
\begin{proof}
    Notice that, by Hölder inequality,
    \[\begin{split}|\widetilde V_\eps(f)(x)|^p &= \left(\int_\G\frac{|f(x\cdot h)-f(x)|}{N(h)}|\nabla_\G N(h)|\rho_\eps(h)^\frac{1}{p}\rho_\eps(h)^\frac{p-1}{p}\,dh\right)^p\\ & \leq \int_\G\frac{|f(x\cdot h)-f(x)|^p}{N(h)^p}|\nabla_\G N(h)|^p\rho_\eps(h)\, dh\,\left(\int_\G\rho_\eps(h)\,dh\right)^{p-1}\\ &=\|\rho_\eps\|_{L^1(\G)}^{p-1}\int_\G\frac{|f(x\cdot h)-f(x)|^p}{N(h)^p}|\nabla_\G N(h)|^p\rho_\eps(h)\, dh.
    \end{split}\]
    Integrating the previous inequality with respect to $x\in\G$ and applying Lemma \ref{firstestimate}, we get the conclusion.
\end{proof}

\begin{remark}\label{comparIIstar} Note that, for each measurable function $f:\,\G\to\R$, it holds that
\begin{equation}\label{comparIIstarfor}
I_{\eps,p}(f)\le\,\||\nabla_\G N|\|_{L^\infty(\G)}\, I^*_{\eps,p}(f)\,.
\end{equation}
The reverse comparison between $I_{\eps,p}(f)$ and $I^*_{\eps,p}(f)$ may fail, since $|\nabla_\G N|$ may vanish on some subset of $\G$, as pointed out in Example \ref{H1}.
\end{remark}
As in the Euclidean case \cite{BBM} (see also \cite{mengesha}), both energies $I_{\eps,p}(f)$ and $I^*_{\eps,p}(f)$  of 
horizontal Sobolev functions and functions of bounded variation are controlled by the norm of the distributional gradient. Combining with the previous lemma, we get a bound for the $L^p$-norm of the nonlocal gradients.
\begin{lemma}\label{thirdestimate}
    Let $(\rho_\eps)_\eps$ satisfy \eqref{P1} and \eqref{P2}. Then
    \begin{enumerate}
    \item[(i)] if $f\in W^{1,p}_\G(\G)$,  $1\leq p<+\infty$, then 
   
    \[I^*_{\eps,p}(f)\leq C(N)\|\rho_\eps\|_{L^1(\G)}\,\|\nabla_\G f\|^p_{L^p(\G)}\,,\]

    In particular $V_\eps(f)\in L^p(\G)$ and
    \[\|V_\eps(f)\|_{L^p(\G)}\leq 
    C(N)\,\|\rho_\eps\|_{L^1(\G)}\,\|\nabla_\G f\|_{L^p(\G)}.\]
    \item [(ii)] if $f\in BV_\G(\G)$, then
   
     \[I^*_{\eps,1}(f)\leq C(N)\|\rho_\eps\|_{L^1(\G)}\,|D_\G f|(\G)\,,\]
     
    In particular $V_\eps(f)\in L^1(\G)$ and
    \[\|V_\eps(f)\|_{L^1(\G)}\leq 
    C(N)\,\|\rho_\eps\|_{L^1(\G)}\,|D_\G f|(\G).\]
\end{enumerate}
\end{lemma}
 \begin{proof}
     \begin{enumerate}
         \item[(i)] The first estimate can be reached by arguing as in \cite[Proposition 3.5]{barbieri}. We report here a more accurate proof.  
         Assume first $f\in C^1_\G(\G)$.
         For every $h\in\G$, let $\gamma_h:[0,1]\to\G$ be a geodesic (w.r.t the CC-distance) connecting 0 and $h$. 
         It is always possible to assume that $\gamma_h$ is parameterized with constant speed: in particular $| \gamma_h'(t)|_{\gamma_h(t)}=\|h\|_c$ for almost every $t\in[0,1]$. For $x\in\G$, define also $v(t)=v_{x,h}(t):=f(x\cdot \gamma_h(t))$. By the chain rule, since $t\mapsto x\cdot \gamma_h(t)$ is horizontal (by left-invariance of the generating vector fields), then for a.e. $t\in[0,1]$
          \[\dot v(t)=\langle\nabla_\G f(x\cdot \gamma_h(t)),(x\cdot\gamma_h(t))'\rangle_{x\cdot \gamma_h(t)},\]
  where $(x\cdot\gamma_h(t))'\in H\G_{x\cdot\gamma_h(t)}$ denotes the derivative of the curve $t\mapsto x\cdot \gamma_h(t)$ at time $t$. 
      Hence we get, using the fact that $|(x\cdot\gamma_h(t))'|_{x\cdot \gamma_h(t)}=|\gamma_h'(t)|_{\gamma_h(t)}=\|h\|_c$ for a.e. $t\in[0,1]$ (again by invariance of the generating vector fields),   
         \begin{equation}\label{ftc}|f(x\cdot h)-f(x)|\leq \int_0^1|\nabla_\G f(x\cdot \gamma_h(t))||(x\cdot\gamma_h(t))'|\,dt= \|h\|_c\int_0^1|\nabla_\G f(x\cdot \gamma_h(t))|\,dt.\end{equation}
  Integrating on $\G$, we obtain
  \[\begin{split}\int_\G|f(x\cdot h)-f(x)|^p\,dx&\leq \|h\|_c^p\int_\G\left(\int_0^1|\nabla_\G f(x\cdot \gamma_h(t))|\,dt\right)^p\,dx\\ &\leq \|h\|_c^p\int_0^1\int_\G|\nabla_\G f(x\cdot \gamma_h(t))|^p\,dx\,dt\\&\leq \|h\|_c^p\int_0^1\int_\G|\nabla_\G f(y)|^p\,dy\,dt\\&\leq \|h\|_c^p\int_\G|\nabla_\G f(y)|^p\,dy.
  \end{split}\]
 Since $N$ is an arbitrary homogeneous norm on $\G$,  then it is equivalent to $\|\cdot\|_c$. Thus, we obtain that
  \[
  \int_\G|f(x\cdot h)-f(x)|^p\,dx\le\, C(N)\,N(h)^p\,\int_\G|\nabla_\G f(y)|^p\,dy\,.
  \]By an approximation argument, thanks to Theorem \ref{meyersserrin}, the same estimate holds for functions in $W^{1,p}_\G(\G)$. 
  Finally, multiplying both sides by $\rho_\eps(h)/N(h)^p$ and integrating over $\G$, we conclude that 
\[\begin{split}I^*_{\eps, p}(f)&=\int_\G\int_\G\frac{|f(x\cdot h)-f(x)|^p}{N(h)^p}\rho_\eps(h)\,dx\,dh\\&\leq C(N)\int_\G\rho_\eps(h)\,dh\int_\G|\nabla_\G f(y)|^p\,dy=C(N)\|\rho_\eps\|_{L^1(\G)}\|\nabla_\G f\|^p_{L^p(\G)}.\end{split}\]

The final estimate can be obtained by applying Lemma \ref{secondestimate}: 
         \[\begin{split}\|V_\eps(f)\|_{L^p(\G)}&\leq 
         Q\,\|\rho_\eps\|_{L^1(\G)}^{(p-1)/p}\,I_{\eps, p}(f)^{1/p}\leq \\& \leq 
         C(N)\,\|\rho_\eps\|_{L^1(\G)}\,\|\nabla_\G f\|_{L^p(\G)}.\end{split}\]
         \item[(ii)] By density of $C^\infty(\G)\cap W^{1,1}_\G(\G)=C^\infty(\G)\cap BV_\G(\G) $ with respect to the strict convergence in $BV_\G(\G)$ (Theorem \ref{anzellotti}), for every $f\in BV_\G(\G)$ there exists a sequence $(f_h)_h\subset C^\infty(\G)\cap W^{1,1}_\G(\G)$ such that
         \[f_h\to f\quad\text{in } L^1(\G)\] and 
         \[|D_\G f_h|(\G)=\|\nabla_\G f_h\|_{L^1(\G)}\to|D_\G f|(\G).\]
       Up to a subsequence we can also assume that
       \[f_h\to f\quad \text{a.e. in }\G.\]
       Using Fatou's lemma and applying (i) with $p=1$ and $f\equiv f_h$,
       \[I^*_{\eps,1}(f)\leq \liminf_{h\to\infty}I^*_{\eps,1}(f_h)\leq C(N)\|\rho_\eps\|_{L^1(\G)}\liminf_{h\to\infty}\|\nabla_\G f_h\|_{L^1(\G)}=C(N)\|\rho_\eps\|_{L^1(\G)}|D_\G f|(\G).\]
       Applying Lemma \ref{secondestimate}, we get the final estimate.\qedhere
     \end{enumerate}
 \end{proof}
Let $\tilde \rho_\eps$ be the map in \eqref{P2}, namely the profile of $\rho_\varepsilon$. Set 
 \begin{equation}\label{Keps}K_\eps(h):=Q\int_{N(h)}^{+\infty}\frac{\tilde \rho_\eps(t)}{t}\,dt\quad\text{for } h\in\G.\end{equation}
 The next lemma enlightens some important properties of the map $K_\eps$, which we will exploit in the following results. In particular, identity (iv) turns out to be the crucial point in the proof of Theorem \ref{convolution}.
 Recall that, since any two left invariant and homogeneous norms are equivalent, the space ${\rm Lip}_\G(\Omega)$ does not depend on the reference distance we choose. 
 \begin{lemma}\label{propertiesofKeps}
     Let $(\rho_\eps)_\eps$ satisfy \eqref{P1} and \eqref{P2}. Then
     \begin{enumerate}
         \item [(i)] $K_\eps \in L^1(\G)$ and \[\|K_\eps\|_{L^1(\G)}=\|\rho_\eps\|_{L^1(\G)};\]
         \item[(ii)] For every $h\in\G\setminus\{0\}$, it holds $ K_\eps (h)<+\infty$. Moreover $K_\eps$ is continuous on $\G\setminus\{0\}$ and there exists \[\exists\,\lim_{h\to 0} K_\eps (h)\in[0,+\infty];\]
         \item[(iii)] If {\rm{spt}}$(\rho_\eps)\subset B(R_\eps)$, then also {\rm{spt}}$(K_\eps)\subset B(R_\eps)$;
         \item[(iv)] If $\rho_\eps\in L^\infty(\G)$ and $\rho_\eps$ has compact support in $\G\setminus\{0\}$, then $K_\eps\in \rm{Lip}_{\G,c}(\G)$ and
         \[\nabla_\G K_\eps (h)=-Q\,\frac{\nabla_\G N(h)}{N(h)}\,\rho_\eps(h)\quad\text{for a.e.}\,h\in \G\setminus\{0\}.\]
     \end{enumerate}
    Moreover, if $(\rho_\eps)_\eps$ satisfy \eqref{P1}$\div$\eqref{P4}, then
    \begin{enumerate}
        \item[(v)] $(K_\eps)_\eps$ turns out to be an approximation to the identity.
    \end{enumerate}
     
 \end{lemma}
 \begin{proof}
    \begin{enumerate}
        \item[(i)] It is clear that $K_\eps$ is measurable, radial and nonnegative. Applying Theorem \ref{polarcoordinates} and Fubini-Tonelli theorem
        \[\begin{split}\int_\G K_\eps(h)\,dh&=Q\,\sigma(S_\G)\int_0^{+\infty}r^{Q-1}\int_r^{+\infty}\frac{\tilde\rho_\eps(t)}{t}\,dt\,dr\\&=Q\sigma(S_\G)\int_0^{+\infty}\frac{\tilde\rho_\eps(t)}{t}\int_0^tr^{Q-1}\,drdt\\ &=\sigma(S_\G)\int_0^{+\infty}\tilde\rho_\eps(t)t^{Q-1}\,dt=\int_\G\rho_\eps(h)\,dh.\end{split}\]
        \item[(ii)] Notice that $K_\eps$ is radial and radially decreasing. Therefore, if $K_\eps(h)=+\infty$ for some $h\in\G\setminus\{0\}$, then $K_\eps=+\infty$ on $B(0,N(h))$, a contradiction with (i). Hence $K_\eps(h)<+\infty$. This also implies that $t\to\tilde\rho_\eps(t)/t$ is in $L^1((a,+\infty))$ for any $a>0$: it follows that $K_\eps$ is continuous on $\G\setminus\{0\}$. Finally, being $K_\eps$ radial and radially decreasing, the desired limit exists. 
        \item[(iii)] Assume spt$(\rho_\eps)\subset B(R_\eps)$. Then $\tilde \rho_\eps(t)=0$ for $t> R_\eps$. Then, by definition, $K_\eps(h)=0$ if $h\not\in  B(R_\eps)$. This proves that \rm{spt}$(K_\eps)\subset B(R_\eps)$;
        \item[(iv)] 
        Since $\rho_\eps$ is bounded and has compact support in $\G\setminus\{0\}$, the same holds for $\tilde\rho_\eps$ in the interval $(0,+\infty)$.
        Notice that $K_\eps(h)=(\theta_\eps\circ N)(h)$, where\[\theta_\eps(s):=Q\int_s^{+\infty}\frac{\tilde\rho_\eps(t)}{t}\,dt.\]
        Since $\tilde\rho_\eps\in L^\infty((0,+\infty))$ and has compact support, then the same is true for $t\to \tilde\rho_\eps(t)/t$, and in particular we get $\theta_\eps\in \text{Lip}_c([0,+\infty))$. Finally, using the Lipschitz continuity of the homogeneous norm $N$, we conclude that $K_\eps=\theta_\eps\circ N\in \rm{Lip}_{\G,c}(\G)$.
        \medskip\\
        By Pansu differentiability theorem, $N$ is P-differentiable for a.e. $h\in\G$. Moreover, by an easy consequence of Theorem \ref{polarcoordinates}, $\theta_\eps$ is differentiable in $N(h)$ for almost every $h\in\G$. Select a point (different than 0) with the previous properties. Then we can apply the chain rule and obtain
        \[\nabla_\G K_\eps(h)=\theta'_\eps(N(h))\,\nabla_\G N(h)=-Q\,\frac{\tilde\rho_\eps(N(h))}{N(h)}\nabla_\G N(h)=-Q\,\frac{\nabla_\G N(h)}{N(h)}\,\rho_\eps(h).\]
        \item[(v)] By (i) and property \eqref{P3}, $(K_\eps)_\eps$ is bounded in $L^1(\G)$ and 
        \[\int_\G K_\eps (h)\,dh=1\quad\text{for every }\eps>0.\]
        We only need to show that 
        \[\lim_{\eps\to0}\int_{\{N(h)>\delta\}}K_\eps(h)\,dh=0\quad\text{for all }\delta>0.\]
        Indeed, arguing as in (i),
        \[\begin{split}\int_{\{N(h)>\delta\}}K_\eps(h)\,dh&=Q\,\sigma(S_\G)\int_\delta^{+\infty}r^{Q-1}\int_r^{+\infty}\frac{\tilde\rho_\eps(t)}{t}\,dt\,dr\\ &=Q\sigma(S_\G)\int_\delta^{+\infty}\frac{\tilde\rho_\eps(t)}{t}\int_\delta^tr^{Q-1}\,drdt\\&\leq \sigma(S_\G)\int_\delta^{+\infty}\tilde\rho_\eps(t)t^{Q-1}\,dt=\int_{\{N(h)>\delta\}}\rho_\eps(h)\,dh\stackrel{\eqref{P4}}\to 0.\qedhere\end{split}\]
    \end{enumerate} 
 \end{proof}
\noindent The next result clarifies why nonlocal gradients are approximations of the horizontal gradient. This crucial identity has been pointed out in \cite{brezis} for the Euclidean gradient and we can now extend it for maps defined on Carnot groups.
\begin{theorem}[Representation formula for the nonlocal gradient]\label{convolution}\phantom{}\medskip\\
Assume $(\rho_\eps)_\eps$ satisfy \eqref{P1} and \eqref{P2} and let $K_\eps$ be the function defined in \eqref{Keps}. 
\begin{enumerate}
    \item[(i)] Let $1\leq p<+\infty$ and $f\in W^{1,p}_\G(\G)$. Then
    \[V_\eps(f)(x)=(\nabla_\G f \ast K_\eps)(x)\quad\text{for a.e. } x\in \G.\]
    \item[(ii)]Let $f\in BV_\G(\G)$. Then
    \[V_\eps(f)(x)=(D_\G f\ast K_\eps)(x)\quad\text{for a.e. }x\in \G.\]
  
\end{enumerate}
\end{theorem}

  \begin{proof}
Following \cite[Lemma 1.5]{brezis}, we divide the proof into three steps.\vspace{0.1 cm}\\
\textbf{Step 1}: \textit{Proof of (i) in the case $f\in C^{\infty}(\G)$, $\rho_\eps\in \rm{Lip}_{\G,c}(\G\setminus\{0\})$.}\vspace{0.1 cm}\\ In this case, we actually prove that
\[V_\eps(f)(x)=(\nabla_\G f \ast K_\eps)(x)\quad\text{for every } x\in \G.\]
By Lemma \ref{propertiesofKeps}, we know that $K_\eps\in \rm{Lip}_{\G,c}(\G)$ and \begin{equation}\label{gradientofKeps}\nabla_\G K_\eps (h)=-Q\,\frac{\nabla_\G N(h)}{N(h)}\,\rho_\eps(h)\quad\text{for a.e.}\,h\in \G\setminus\{0\}.\end{equation}
Let's now notice that, for each $\eps>0$ and for each $x\in\G$,
\begin{equation}\label{byparts}
    \int_\G\nabla_\G f(x\cdot h) K_\eps(h)\,dh=-\int_\G (f(x\cdot h)-f(x))\nabla_\G K_\eps(h)\,dh.
\end{equation}
Indeed, for every $\eps>0$, it holds $K_\varepsilon\in \text{Lip}_{\G,c}(\G)$ by Lemma \ref{propertiesofKeps} (iv). Since Lip$_{\G,c}(\G)\subset W^{1,\infty}_\G(\G)$ and $K_\eps$ has compact support in $\G$, \eqref{byparts} follows by an integration by parts and the left-invariance of $\nabla_\G$. 
Moreover, since $K_\eps(h)=K_\eps(h^{-1})$, by a change of variable we get  
\begin{equation}\label{convchange}\int_\G\nabla_\G f(x\cdot h)K_\eps(h)\,dh=   \int_\G\nabla_\G f(x\cdot h^{-1})K_\eps(h)\,dh = (\nabla_\G f\, \ast\,K_\eps)(x).\end{equation}
Combining \eqref{gradientofKeps}, \eqref{byparts} and \eqref{convchange}, we get the conclusion. \medskip\\
\textbf{Step 2}: \textit{Proof of (i) and (ii) in the case $\rho_\eps\in\rm{Lip}_{\G,c}(\G\setminus\{0\})$ and $f\in W^{1,p}_\G(\G)$ or $f\in BV_\G(\G)$.}\vspace{0.1 cm}\\ 
Let $(\eta_\delta)_{\delta>0}$ be a family of intrinsic mollifiers. First consider the case of $f\in W^{1,p}_\G(\G)$. By Step 1 and Lemma \ref{propconv1},
\[V_\eps(\eta_\delta\ast f)=\nabla_\G(\eta_\delta\ast f)\ast K_\eps=(\eta_\delta\ast\nabla_\G f)\ast K_\eps=\eta_\delta\ast(\nabla_\G f\ast K_\eps).\]
By Proposition \ref{young} and Lemma \ref{propertiesofKeps} (iv) one gets $\nabla_\G f\ast K_\eps\in L^p(\G)$: hence, by Lemma \ref{propconv1}, $\eta_\delta\ast(\nabla_\G f\ast K_\eps)\to\nabla_\G f\ast K_\eps$ in $L^p(\G)$ as $\delta\to 0$. In particular the convergence holds almost everywhere, up to passing to a discrete subsequence.\medskip\\
Similarly, if $f\in BV_\G(\G)$, by applying Step 1 and Lemma \ref{propconv2}, we get
\[V_\eps(\eta_\delta\ast f)=\nabla_\G(\eta_\delta\ast f)\ast K_\eps=(\eta_\delta\ast D_\G f)\ast K_\eps=\eta_\delta\ast(D_\G f\ast K_\eps).\]
Recalling that, by Lemma \ref{propertiesofKeps}, $K_\eps\in \rm{Lip}_{\G,c}(\G)$, it follows from Lemma \ref{propconv2} that $D_\G f\ast K_\eps\in C^0_b(\G)$. Hence $\eta_\delta\ast(D_\G f\ast K_\eps)\to D_\G f\ast K_\eps$ uniformly on compact sets by Lemma \ref{propconv1}. \medskip\\
In order to reach the desired conclusion, it suffices to show that
\begin{equation}\label{Vepsconvergence} V_\eps(\eta_\delta\ast f)\stackrel{\delta\to 0^+}\longrightarrow V_\eps(f)\end{equation}almost everywhere in $\G$ either if $f\in W^{1,p}_\G(\G)$ or $f\in BV_\G(\G)$.
Since $f\in L^1_{\rm loc}(\G)$, then $\eta_\delta\ast f\to f$ in $L^{1}_{\rm loc}(\G)$: hence we can assume that $f_\delta(x):=(\eta_\delta\ast f)(x)\to f(x)$ for a.e. $x\in\G$. Then \eqref{Vepsconvergence} follows by Lebesgue's dominated convergence theorem if we prove the existence, for a.e. $x\in\G$, of a function $g=g_{x,\eps}\in L^1(\G)$ such that
\[\frac{|f_\delta(x\cdot h)-f_\delta(x)|}{N(h)}|\nabla_\G N(h)|\rho_\eps(h)\leq g(h)\quad\text{for a.e. }h\in\G,\;\;\text{for all }\delta\in(0,1).\]
In order to construct $g$, define the truncated \textit{right} maximal function $M_{R,1} f:\G\to[0,+\infty]$ as
\[M_{R,1} f(x):=\sup_{0<\delta\leq 1} \fint_{B(\delta)\cdot x} |f|\,d\mathcal L^n.\]
It follows from standard arguments that $M_{R,1} f\in L^p_{\rm loc}(\G)$ whenever $f\in L^p_{\rm loc}(\G)$ for $p>1$ (see for instance \cite[Theorem 1]{stein}, noticing that the family $\{B(\delta)\cdot x:x\in\G,\,\delta>0\}$ satisfies assumptions $(i)\div(iv)$ in the same reference). Moreover, an explicit computation shows that $|f_\delta(x)|=|(\eta_\delta\ast f)(x)|\leq M_{R,1}f(x)$ for every $\delta\in(0,1)$:
\[\begin{split} |(\eta_\delta\ast f)(x)|&=\left|\int_\G\eta_\delta(x\cdot y^{-1})f(y)\,dy\right|\\&\leq \int_{B(\delta)\cdot x}|\eta_\delta(x\cdot y^{-1})||f(y)|\,dy\\&
\lesssim \fint_{B(\delta)\cdot x}|f(y)|\,dy\leq M_{R,1} f(x),
    \end{split}\]
    where, with the notation $a\lesssim b$, we denote, here and in the following, the existence of a positive constant $C$ such that $a\leq Cb$.
Hence we obtain the domination
\[\frac{|f_\delta(x\cdot h)-f_\delta(x)|}{N(h)}|\nabla_\G N(h)|\rho_\eps(h)\leq C(M_{R,1} f(x\cdot h)+M_{R,1}f(x))\frac{\rho_\eps (h)}{N(h)}=:g(h).\]
Let's now show that $g\in L^1(\G)$. By the extra assumption on $\rho_\eps$ (i.e. $\rho_\eps\in {\rm Lip}_{\G,c}(\G\setminus\{0\}))$, the function \[h\mapsto\frac{\rho_\eps(h)}{N(h)}\in C^0_c(\G\setminus\{0\}).\] 
In addition, by Sobolev-Poincaré inequality (see, for instance, \cite{lu92}, \cite{franchi95}, \cite{saloff-coste}), 
$f\in L^{q}_{\rm loc}(\G)$ for all $1<q\leq\frac{Q}{Q-1}$, which also implies \[M_{R,1} f\in L^{q}_{\rm loc}(\G)\subset L^1_{\rm loc}(\G).\]
In particular $M_{R,1} f<+\infty$ almost everywhere in $\G$ and, by left-invariance of the Lebesgue measure, also $h\mapsto M_{R,1} f(x\cdot h)\in L^1_{\rm loc}(\G)$. Therefore \eqref{Vepsconvergence} holds on the set 
\[E:=\{x\in\G: M_{R,1}f(x)<+\infty,\, f_\delta(x)\stackrel{\delta\to 0}\rightarrow f(x)\}.\]
\noindent
\textbf{Step 3}: \textit{Proof of (i) and (ii) in the general case.}\vspace{0.1 cm}\\ For given $\eps$, we approximate $\rho_\eps$ in $L^1(\G)$ with a sequence $(\rho_{\eps,j})_j\subset \rm{Lip}_{\G,c}(\G\setminus\{0\})$ satisfying \eqref{P1} and \eqref{P2}. We construct such a sequence in the following way. For each $j\in\mathbb N$, let $\psi_j(r):=r^{Q-1}\chi_{(\frac{1}{j},j)}(r)$ and $\rho_{\eps,j,\delta}^*(r):=((\tilde\rho_\eps\psi_j)\ast \eta_\delta)(r)$, where $(\eta_\delta)_\delta$ denotes a family of (classical) mollifiers on $\R$.
Then define
\[\widetilde \rho_{\eps,j,\delta}(r):=\frac{\rho_{\eps,j,\delta}^*(r)}{r^{Q-1}}\quad\text{if }r\in(0,+\infty)\]
and
\[\rho_{\eps,j,\delta}(x):=\widetilde\rho_{\eps,j,\delta}(N(x))\quad\text{if } x\in\G.\]
By construction, $\rho_{\eps,j,\delta}$ are radial and positive.
By Theorem \ref{polarcoordinates} note that, since $\rho_\eps\in L^1(\G)$, 
\begin{equation}\label{integrability}\sigma(S_\G)\int_0^{+\infty}\tilde\rho_\eps(r)\,r^{Q-1}\,dr=\int_\G\rho_\eps(x)\,dx<+\infty.\end{equation}
Therefore, for every $j\in\mathbb N$, $\tilde\rho_\eps \psi_j\in L^1((0,+\infty))$. Then, by the classical approximation by convolution properties, $\rho_{\eps,j,\delta}^*\to \tilde\rho_\eps\psi_j$ in $L^1$ when $\delta\to 0$ and also ${\rm{spt}}(\rho^*_{\eps,j,\delta})\subset[\frac{1}{j},j]+[-\delta,\delta]$. For given $\eps>0, j\in\mathbb N$, fix $\bar\delta=\bar\delta(\eps, j)<1/j$ such that \begin{equation}\label{1/j}
\int_0^{+\infty}|\rho_{\eps,j,\delta}^*-\tilde\rho_\eps\psi_j|\,dr<\frac{1}{j}\quad\text{for every }\delta\leq\bar\delta.
\end{equation}
Denote $\rho^*_{\eps,j}:=\rho_{\eps,j,\bar\delta}^*$, $\widetilde \rho_{\eps,j}:=\widetilde\rho_{\eps,j,\bar\delta}$ and $\rho_{\eps,j}:=\rho_{\eps,j,\bar\delta}$.
We compute, again by Theorem \ref{polarcoordinates},
\begin{equation}\begin{aligned}\label{esti1}\int_\G|\rho_{\eps,j}(x)-\rho_\eps(x)|\,dx&=\sigma(S_\G)\int_0^{+\infty}|\widetilde\rho_{\eps,j}(r)-\tilde\rho_\eps(r)|r^{Q-1}\,dr\\&=\sigma(S_\G)\int_0^{+\infty}|\rho^*_{\eps,j}(r)-\tilde\rho_\eps(r)r^{Q-1}|\,dr.\end{aligned}\end{equation}
From \eqref{1/j} we get 
\begin{equation}\begin{aligned}\label{esti2}\int_0^{+\infty}|\rho_{\eps,j}^\ast(r)-\tilde\rho_\eps(r)r^{Q-1}|\,dr&\leq \int_0^{+\infty}|\rho_{\eps,j}^*-\tilde\rho_\eps\psi_j|\,dr+\int_0^{+\infty}|\tilde\rho_\eps\psi_j-\tilde\rho_\eps(r)r^{Q-1}|\,dr\\ &\leq \frac{1}{j}+\int_0^{+\infty}|\tilde\rho_\eps\psi_j-\tilde\rho_\eps(r)r^{Q-1}|\,dr.\end{aligned}\end{equation}
By \eqref{integrability}, the map $r\to \tilde\rho_\eps(r)\,r^{Q-1}$ is in $L^1((0,+\infty))$: hence, by dominated convergence, \begin{equation}\label{esti3}\tilde\rho_\eps\psi_j\to \tilde\rho_\eps\,r^{Q-1}\quad\text{in } L^1((0,+\infty)).\end{equation}
Combining \eqref{esti1}, \eqref{esti2} and \eqref{esti3}, we deduce that $\rho_{\eps,j}\to\rho_\eps$ in $L^1(\G)$ as $j\to\infty$.\medskip\\
Now we show that $\rho_{\eps,j}\in\rm{Lip}_{\G,c}(\G\setminus\{0\}).$ Since ${\rm{spt}}(\widetilde \rho_{\eps, j})={\rm{spt}}( \rho^*_{\eps, j})\subset[\frac{1}{j}-\bar\delta,\,j+\bar\delta]$, it follows that $\widetilde\rho_{\eps,j}\in C^{\infty}_c((0,+\infty))$.
For each $x,y\in\G$
\[|\rho_{\eps,j}(x)-\rho_{\eps,j}(y)|=|\widetilde\rho_{\eps,j}(N(x))-\widetilde\rho_{\eps,j}(N(y))|\leq {\rm{Lip}}(\widetilde \rho_{\eps,j})|N(x)-N(y)|\leq {\rm{Lip}}(\widetilde \rho_{\eps,j})\,d_N(x,y).\]
Moreover, since ${\rm{spt}}(\widetilde\rho_{\eps, j})\subset[\frac{1}{j}-\bar\delta,\,j+\bar\delta]\Subset(0,+\infty)$, this implies ${\rm{spt}}(\rho_{\eps, j})\Subset\G\setminus\{0\}$. \medskip\\
Define
\[V_{\eps,j}(f)(x):=Q\int_\G \frac{f(x\cdot h)-f(x)}{N(h)}\nabla_\G N(h)\rho_{\eps,j}(h)\,dh\]
and let
\[K_{\eps,j}(h):=Q\int_{N(h)}^{+\infty}\frac{\widetilde \rho_{\eps,j}(t)}{t}\,dt\quad\text{for } h\in\G\setminus\{0\}.\]
Since $\rho_{\eps,j}\in\rm{Lip}_{\G,c}(\G\setminus\{0\})$ satisfies \eqref{P1} and \eqref{P2}, we can apply Step 2 and we get, for given $\eps>0,\,j\in\mathbb N$, that
\[V_{\eps,j}(f)=\begin{cases}\nabla_\G f\ast K_{\eps,j}&\text{a.e. in }\G, \;\;\text{if    }f\in W^{1,p}_\G(\G)\\D_\G f\ast K_{\eps,j}&\text{a.e. in }\G, \;\;\text{if  }f\in BV_\G(\G).
\end{cases}\]
The conclusion will be achieved by passing to the limit as $j\to +\infty$ if we prove that
\begin{equation}\label{limit1}
    \lim_{j\to\infty}\|V_{\eps,j}(f)-V_\eps(f)\|_{L^p(\G)}=0\quad\text{for every }f\in W^{1,p}_\G(\G),
\end{equation}
\begin{equation}\label{limit2}
    \lim_{j\to\infty}\|V_{\eps,j}(f)-V_\eps(f)\|_{L^1(\G)}=0\quad\text{for every }f\in BV_\G(\G),
\end{equation}
\begin{equation}\label{limit3}
    \lim_{j\to\infty}\|\nabla_\G f\ast K_{\eps,j}-\nabla_\G f\ast K_\eps\|_{L^p(\G)}=0 \quad\text{for every }f\in W^{1,p}_\G(\G),
\end{equation}
\begin{equation}\label{limit4}
    \lim_{j\to\infty}\|D_\G f\ast K_{\eps,j}-D_\G f\ast K_\eps\|_{L^1(\G)}=0 \quad\text{for every }f\in BV_\G(\G).
\end{equation}
Properties \eqref{limit1} and \eqref{limit2} follow from Lemma \ref{thirdestimate}. Then observe that $K_{\eps,j}\to K_\eps$ in $L^1(\G)$. Indeed, arguing as in Lemma \ref{propertiesofKeps},
\[\begin{split}\|K_{\eps,j}-K_\eps\|_{L^1(\G)}&=Q\int_\G\left|\int_{N(h)}^{+\infty}\frac{\widetilde \rho_{\eps,j}(t)-\tilde\rho_\eps(t)}{t}\,dt\right|\,dh\\&=Q\sigma(S_\G)\int_0^{+\infty}\left|\int_{r}^{+\infty}\frac{\widetilde \rho_{\eps,j}(t)-\tilde\rho_\eps(t)}{t}\,dt\right|r^{Q-1}\,dr\\ &\leq Q\sigma(S_\G)\int_0^{+\infty}\int_{r}^{+\infty}\frac{|\widetilde \rho_{\eps,j}(t)-\tilde\rho_\eps(t)|}{t}r^{Q-1}\,dt\,dr\\ &= Q\sigma(S_\G)\int_0^{+\infty}\int_{0}^{t}\frac{|\widetilde \rho_{\eps,j}(t)-\tilde\rho_\eps(t)|}{t}r^{Q-1}\,dr\,dt\\ &=\sigma(S_\G)\int_0^{+\infty}|\rho_{\eps,j}^*(t)-\tilde\rho_\eps(t)t^{Q-1}|\,dt.\end{split}\]
By \eqref{esti2} and \eqref{esti3}, we get the desired convergence.\medskip\\
By Young inequality (Proposition \ref{young}), if $f\in W^{1,p}_\G(\G)$,
\[\|\nabla_\G f\ast K_{\eps,j}-\nabla_\G f\ast K_\eps\|_{L^p(\G)}=\|\nabla_\G f\ast (K_{\eps,j}-K_\eps)\|_{L^p(\G)} \leq \|\nabla_\G f\|_{L^p(\G)}\|K_{\eps,j}-K_\eps\|_{L^1(\G)}.\]
Hence \eqref{limit3} follows. Let's finally show \eqref{limit4}.
\[\begin{split}|(D_\G f\ast K_{\eps,j})(x)-(D_\G f\ast K_\eps) (x)|&=|(D_\G f\ast(K_{\eps,j}-K_\eps))(x)|\\&\leq \int_\G|K_{\eps,j}(y^{-1}x)-K_\eps(y^{-1}x)|\,d|D_\G f|(y).\end{split}\]
By integrating the previous inequality on $\G$ and by using the left-invariance of the Lebesgue measure, it follows that
\[\begin{split}\|D_\G f\ast K_{\eps,j}-D_\G f\ast K_\eps\|_{L^1(\G)}&\leq \int_\G\int_\G|K_{\eps,j}(y^{-1}x)-K_\eps(y^{-1}x)|\,d|D_\G f|(y)\,dx\\&= \int_\G\int_\G|K_{\eps,j}(y^{-1}x)-K_\eps(y^{-1}x)|\,dx\,d|D_\G f|(y)\\&=\|K_{\eps,j}-K_\eps\|_{L^1(\G)}|D_\G f|(\G).\end{split}\]
Letting $j\to\infty$ in the previous inequality, \eqref{limit4} follows.
\end{proof}
    As a corollary of Theorem \ref{convolution}, we immediately deduce the following central approximation result for the horizontal gradient. Notice that, differently from the case involving $I^\ast_{\eps,p}(f)$ (see \eqref{barbieribbm}), there is no positive constant appearing in the limit and, moreover, the convergence holds without any further assumption on the homogeneous norm $N$.
\begin{corollary}[Nonlocal approximation of the horizontal gradient]\label{approximation}\phantom{}\vspace{0.1 cm}\\
    Assume $(\rho_\eps)_\eps$ is a family of mollifiers on $\G$ satisfying \eqref{P1}$\div$\eqref{P4}.
    \begin{enumerate}
        \item [(i)] Let $f\in W^{1,p}_\G(\G)$ for $1\leq p<+\infty.$ Then
       \[V_\eps(f)\longrightarrow\nabla _\G f\quad\text{in }L^p(\G).\]
      In particular,
       \[
       \|V_\eps(f)\|_{L^p(\G)}\to \||\nabla_\G f|\|_{L^p(\G)}\,.
       \]
        \item[(ii)] Let $f\in BV_\G(\G)$. Then 
       \[V_\eps(f)\rightharpoonup D_\G (f)\quad\text{in }\mathcal M(\G,\R^{m_1}),\]
      and 
       \[
       \|V_\eps(f)\|_{L^1(\G)}\to |D_\G f|(\G)\,.\]
   \end{enumerate}
\end{corollary}
    \begin{proof}
 \begin{enumerate}
     \item [(i)] From Lemma \ref{propconv1}, Lemma \ref{propertiesofKeps} (v) and Theorem \ref{convolution}  
     \[V_\eps(f)=(\nabla_\G f\ast K_\eps)\to \nabla_\G f\quad\text{in }L^p(\G);\]
     \item [(ii)] From Lemma \ref{propconv2}, Lemma \ref{propertiesofKeps} (v) and Theorem \ref{convolution}  
     \[V_\eps(f)=(D_\G f\ast K_\eps)\rightharpoonup D_\G f\quad\text{in }\mathcal M(\G,\R^{m_1})\]
     and 
     \[
       \|V_\eps(f)\|_{L^1(\G)}\to |D_\G f|(\G)\,.\qedhere
       \]
 \end{enumerate}  
 
    \end{proof}
\begin{remark} In addition to Corollary \ref{approximation}, we can actually prove pointwise convergence almost everywhere. 
More precisely, it holds that \[\begin{split}V_\eps(f)\to\nabla_\G f \quad\,&\text{ a.e. if }\,f\in W^{1,p}_\G(\G), \\V_\eps(f)\to \nabla_\G^\text{ac} f\quad&\text{ a.e. if }\,f\in BV_\G(\G).\end{split}\]

The proof can be carried out as in the Euclidean case (see \cite[Proposition 1.10 and 1.11]{brezis}), using integration in polar coordinates (Theorem \ref{polarcoordinates}) and Lebesgue-Besicovitch differentiation theorem for singular measures on Carnot groups (see \cite[Lemma 2.1]{magnani}).

In addition (see also \cite[Remark 1.13]{brezis}), being $K_\eps$ an approximation to the identity, it follows by a slight improvement of Lemma \ref{propconv1} (ii) that   
\[\begin{split}V_\eps(f)\to\nabla_\G f \quad\,&\text{ uniformly on $\G$, if }\,f\in C^1_{\G,c}(\G), \\V_\eps(f)(x)\to \nabla_\G f(x)\quad&\text{ for each $x\in\G$, if }\,f\in C^1_{\G,b}(\G)\text{ and }\nabla_\G f\in L^\infty(\G).\end{split}\]

\end{remark}
\begin{remark}\label{gradeuclidapprox}
   Notice that $V_\eps(f)$ depends on the chosen homogeneous norm $N$, while the limit value $\nabla_\G f$ (or $D_\G f$) does not. Moreover, our results apply also to the Euclidean space $\G=\R^n$ equipped with an arbitrary norm $N$: hence we actually extend the case discussed in \cite{brezis} for the standard Euclidean norm. 
   Note that, in this case, the horizontal gradient $\nabla_\G\equiv\nabla$, where $\nabla$ denotes the classical Euclidean gradient, and the nonlocal gradient has the form
   \[
   V_\eps(f)(x):= n\int_{\R^n} \frac{f(x+h)-f(x)}{N(h)}\nabla N(h)\rho_\eps(h)\,dh\,.
   \]
\end{remark}
    We are now interested in a partial converse of Corollary \ref{approximation}: knowing that $V_\eps(f)$ are equibounded in $L^p(\G)$, then one can infer that $f\in W^{1,p}_\G(\G)$ (or $f\in BV_\G(\G)$ in the case $p=1$). However, under the assumptions $\eqref{P1}\div\eqref{P4}$, $V_\eps(f)$ might not be well-defined for a general $f\in L^p(\G)$. Hence, for the time being, we will assume a more restrictive condition on the family of mollifiers $(\rho_\eps)_\eps$, which is motivated by Theorem \ref{convolution}, Steps 1-2:
\begin{equation}\label{P5}\tag{P5}
        \rho_\eps\in L^\infty(\G), \quad \text{spt}(\rho_\eps)\subset B(0,R)\setminus B(0,r_\eps) \text{ for some }0<r_\eps<R.
    \end{equation}
\begin{lemma}\label{Vepswelldefined}
    Assume $(\rho_\eps)_\eps$ satisfy \eqref{P1} and \eqref{P5}. Then, for every $f\in L^{1}_{\rm loc}(\G)$, $\widetilde V_\eps(f)\in L^1_{\rm loc}(\G)$. In particular, by Lemma \ref{firstestimate}, $V_\eps(f)$ is well-defined a.e. and $V_\eps(f)\in L^1_{\rm loc}(\G)$.
\end{lemma}
\begin{proof}
    Let $\Omega\Subset\G$. It follows by an easy computation that
   
        \begin{align*}
            \int_\Omega\widetilde V_\eps(f)(x)\,dx&=\int_\Omega\int_\G\frac{|f(x\cdot h)-f(x)|}{N(h)}|\nabla_\G N(h)|\,\rho_\eps(h)\,dh\,dx
            \\ & \leq \frac{\|\nabla_\G N\|_{L^\infty(\G)}}{r_\eps}(\|f\|_{L^1(\Omega)}+\|f\|_{L^1(\Omega\cdot R_\eps)})\int_\G\rho_\eps(h)\,dh<+\infty.\qedhere
        \end{align*}
\end{proof}
We now show that the (weak) convergence of the nonlocal gradients $V_\eps(f)$ in $L^p(\G)$ is also a sufficient condition for a function $f$ to lie in $W^{1,p}_\G(\G)$ or $BV_\G(\G)$.
\begin{lemma}\label{converseimplication}
    Assume $(\rho_\eps)_\eps$ satisfy conditions \eqref{P1}$\div$\eqref{P5}. Let $f\in L^p(\G)$, for $1\leq p<+\infty$. \begin{itemize}
        \item Suppose there exists $w\in (L^p(\G))^{m_1}$ such that \[V_\eps(f)\rightharpoonup w\quad \text{ in } (L^p(\G))^{m_1}.\] Then $f\in W^{1,p}_\G(\G)$ and $\nabla_\G f=w$ a.e. in $\G$.
        \item Suppose there exists a finite Radon measure $\mu\in \mathcal M(\G,\R^{m_1})$ such that \[V_\eps(f)\rightharpoonup \mu\quad\text{ in } \mathcal M(\G,\R^{m_1}).\] Then $f\in BV_\G(\G)$ and $D_\G f=\mu$.
    \end{itemize}
\end{lemma}
\begin{proof}
  It follows from Lemma \ref{Vepswelldefined} that $V_\eps(f)$ is well defined. By Lemma \ref{propertiesofKeps} (iv) and a change of variable, we can write
  \[V_\eps(f)(x)=-\int_\G f(x\cdot h)\nabla_\G K_\eps(h)\,dh=-\int_\G f(y)\nabla_\G K_\eps(x^{-1}y)\,dy.\]
  By Fubini-Tonelli theorem, for every $\varphi\in C^{\infty}_c(\G)$
  \[\int_\G V_\eps(f)(x)\varphi(x)\,dx=-\int_\G f(y)\left(\int_\G\nabla _\G K_\eps(x^{-1}y)\varphi(x) \,dx\right)dy.\]
  If we denote
  \[
\begin{alignedat}{2}
\varphi_y : \mathbb{G}\to \mathbb{R}
&\hspace{2em} \varphi_y(z) := \varphi\bigl(y\cdot z^{-1}\bigr),
\quad z\in\mathbb{G},\\
\Phi : \mathbb{G}\times\mathbb{G}\to \mathbb{R}
&\hspace{2em} \Phi(y,w) := \nabla_{\mathbb{G}} \varphi_y\bigl(y\cdot w^{-1}\bigr),
\quad y,w\in\mathbb{G},\\
\Phi_\varepsilon : \mathbb{G}\to \mathbb{R}
&\hspace{2em} \Phi_\varepsilon(y) := (K_\varepsilon * \Phi(y,\cdot))(y),
\quad y\in\mathbb{G},
\end{alignedat}\]
then, for every $y\in \G$, 
\begin{align*}
    \int_\G\nabla _\G K_\eps(x^{-1}y)\varphi(x) \,dx&=\int_\G\nabla_\G K_\eps (z)\varphi(yz^{-1})\,dz=\int_\G \nabla_\G K_\eps(z)\varphi_y(z)\,dz\\ &=-\int_\G K_\eps(z)\nabla_\G \varphi_y(z)\,dz= -\int_\G K_\eps(yw^{-1})\nabla_\G \varphi_y(yw^{-1})\,dw\\ &=-\int_\G K_\eps(yw^{-1})\Phi(y,w)\,dw=- (K_\eps*\Phi(y,\cdot))(y)=-\Phi_\eps(y).
\end{align*}
Hence we find that
\begin{equation}\label{eq1}\int_\G f(y)\Phi_\eps(y)\,dy=\int_\G V_\eps(f)(x)\varphi(x)\,dx.\end{equation}
The conclusion will follow by passing to the limit in \eqref{eq1}.
Since $\Phi(y,\cdot)\in C^{0}_b(\G)$, it follows that $K_\varepsilon * \Phi(y,\cdot)\to \Phi(y,\cdot)$ pointwise on $\G$ (by Lemma \ref{propertiesofKeps} (v) and Lemma \ref{propconv1} (ii)). In particular, $\Phi_\eps(y)=(K_\varepsilon * \Phi(y,\cdot))(y)\to \Phi(y,y)=\nabla_\G\varphi_y(0)$. Let us now show that 
\[\nabla_\G \varphi_y(0)=-\nabla_\G\varphi (y).\]
First, if $\widetilde\varphi_y:\G\to \R$ is defined by $\widetilde \varphi_y(z):=\varphi(y\cdot z)$, by the left-invariance of $\nabla_\G$ it follows that $\nabla_\G\widetilde\varphi_y(0)=\nabla_\G\varphi(y)$. Second, observe that $\varphi_y(z)=\widetilde\varphi_y(z^{-1})$. Hence, it also holds that $\nabla_\G\varphi_y(0)=-\nabla_\G\widetilde\varphi_y(0)$, since the horizontal vector fields $X_1,\dots, X_{m_1}$ can be taken to coincide with the standard partial derivatives at the origin; see \cite[Proposition 2.2]{FSSC03}.\medskip\\
We finally look for a dominating summable function for passing to the limit on the left-hand side of \eqref{eq1}. Let $\widetilde R$ be such that ${\mathrm{spt}}(\varphi)\subset B(\widetilde R)$. By definition of $\varphi_y$ and $\Phi$, 
\[\mathrm{spt}(\Phi)\subset\{(y,w)\in\G\times\G: ywy^{-1}\in B(\widetilde R)\}=: S_1.\]
If $\widetilde K_\eps(y,w):=K_\eps(yw^{-1})$, then, by property \eqref{P5} and Lemma \ref{propertiesofKeps} (iii),
\[\mathrm{spt}(\widetilde K_\eps)\subset\{(y,w)\in\G\times\G: yw^{-1}\in B(R)\}=: S_2.\]
This implies that the support of the function $\G\times\G\ni(y,w)\mapsto K_\eps(yw^{-1})\Phi(y,w)$ is compact: in fact, by the triangular inequality, $S_1\cap S_2$ is bounded in $\G\times \G$. Let $\overline R>0$ be such that ${\mathrm {spt}}(\widetilde K_\eps\Phi)\subset B(\overline R)\times B( \overline R)$. Hence we can estimate
\[|\Phi_\eps(y)|=|(K_\eps*\Phi(y,\cdot))(y)|\leq \int_{B(\overline R)} K_\eps(yw^{-1})|\Phi(y,w)|\,dw\leq \sup_{B(\overline R)\times B(\overline R)}|\Phi|\int_\G K_\eps(z)\,dz\leq M.\]
Notice here that $\Phi_\eps$ has compact support in $B(\overline R)$ and that $\Phi\in C^0(\G\times \G).$\medskip\\
Since $\Phi_\eps$ are uniformly bounded and have equibounded support, we can apply Lebesgue's dominated convergence theorem and pass to the limit on the left-hand side in \eqref{eq1}. We use instead our weak-convergence hypothesis for the limit on the right-hand side. We get
\[\int_\G f(y)\nabla_\G\varphi(y)\,dy=\begin{cases}
 -\int_\G w(x)\varphi(x)\,dx &\text{in the case }w\in (L^p(\G))^{m_1} \\ -\int_\G\varphi(x)\,d\mu(x) &\text{in the case }\mu\in \mathcal M(\G, \R^{m_1})  
\end{cases}\]
and the conclusion follows.
\end{proof}
\begin{corollary}\label{characterization}
 Let $\G$ be a Carnot group, endowed with a homogeneous norm $N$. Assume $(\rho_\eps)_\eps$ is a sequence of mollifiers satisfying assumptions \eqref{P1}$\div$\eqref{P5}and let $f\in L^p(\G)$\vspace{0.2 cm}. \\ If $1<p<+\infty$,
  \[f\in W^{1,p}_\G(\G)\Longleftrightarrow \limsup_{\eps\to 0}\|V_\eps(f)\|_{L^p(\G)}<+\infty.\]
  If $p=1$, then 
  \[f\in BV_\G(\G)\Longleftrightarrow \limsup_{\eps\to 0}\|V_\eps(f)\|_{L^1(\G)}<+\infty.\]
\end{corollary}
\begin{proof}
    We only discuss the case $1<p<+\infty$. The case $p=1$ is treated in the same way. The implication \[f\in W^{1,p}_\G(\G)\Longrightarrow  \limsup_{\eps\to 0}\|V_\eps(f)\|_{L^p(\G)}<+\infty\] follows by Lemma \ref{thirdestimate}. Assume now $\limsup_{\eps\to 0}\|V_\eps(f)\|_{L^p(\G)}<+\infty$. Then, up to a subsequence, we can assume $V_\eps(f)\rightharpoonup w$ in $(L^p(\G))^{m_1}$. By Lemma \ref{converseimplication}, we deduce that $f\in W^{1,p}_\G(\G)$.
\end{proof} 
We are now in order to show, under assumptions \eqref{P1}$\div$\eqref{P4}, the  characterization of real-valued Sobolev and BV functions, defined on a Carnot group $\G$, with respect to the nonlocal gradients, that is Theorem \ref{firstcharacterization}.

\begin{proof}[Proof of Theorem \ref{firstcharacterization}]
     We only deal with the case $1<p<+\infty$, since the case $p=1$ is treated in the same way. The implications
     \[f\in W^{1,p}_\G(\G)\Longrightarrow \limsup_{\eps\to 0}I^*_{\eps, p}(f)<+\infty\Longrightarrow \limsup_{\eps\to 0}I_{\eps, p}(f)<+\infty\Longrightarrow  
     \limsup_{\eps\to 0}\|\widetilde V_\eps(f)\|_{L^p(\G)}<+\infty\] follow by Lemma \ref{thirdestimate}, \eqref{comparIIstarfor} and Lemma \ref{secondestimate} respectively.

     Assume now \begin{equation}\label{boundednessassumption}\limsup_{\eps\to 0}\|\widetilde V_\eps(f)\|_{L^p(\G)}<+\infty.\end{equation} We approximate $(\rho_\varepsilon)_\varepsilon$ with another sequence of mollifiers $(\widehat\rho_\varepsilon)_\varepsilon\subset L^\infty(\G)$ still satisfying the properties  $\eqref{P1}\div\eqref{P4}$. We construct this sequence in the following way: for every $\varepsilon>0$, let $M_\varepsilon>0$ be such that 
     $\int_\G\min(\rho_\eps, M_\eps)\geq 1/2.$
     We now let \[\widehat \rho_\eps:=\frac{\min(\rho_\eps, M_\eps)}{\|\min(\rho_\eps, M_\eps)\|_{L^1(\G)}}.\]
     It is then clear that $(\widehat\rho_\eps)_\eps$ satisfy $\eqref{P1}\div\eqref{P3}$; we now show $\eqref{P4}$. If $\delta>0$ is fixed, then
     \[\int_{B^c(0,\delta)}\widehat \rho_\eps(h)\,dh=\frac{1}{\|\min(\rho_\eps, M_\eps)\|_{L^1(\G)}}\int_{B^c(0,\delta)}\min(\rho_\eps, M_\eps)\,dh\leq 2\int_{B^c(0,\delta)}\rho_\eps\stackrel{\eps\to 0}\longrightarrow 0.\]
     Using the same argument, it follows that 
     \begin{equation}\label{bound}\int_\G\frac{|f(x\cdot h)-f(x)|}{N(h)}|\nabla_\G N(h)|\,\widehat\rho_\eps(h)\,dh\leq 2\int_\G\frac{|f(x\cdot h)-f(x)|}{N(h)}|\nabla_\G N(h)|\,\rho_\eps(h)\,dh.\end{equation}
     As a second step, we now approximate $(\widehat\rho_\eps)_\eps$ with a further sequence $(\bar\rho_\eps)_\eps$, which also satisfies assumption \eqref{P5}.
     Since $(\widehat\rho_\eps)_\eps$ satisfies \eqref{P4}, we can assume that $\int_{B^c(0,1)}\widehat \rho_\eps\leq 1/4$ for every $\eps>0$. Choose then $m_\eps>0$ small enough, so that $\|\widehat\rho_\eps\|_{L^1(B(0,m_\eps))}\leq 1/4$. We then define 
\[\bar\rho_\eps:=\frac{\widehat\rho_\eps\chi_{C(m_\eps,1)}}{\|\widehat\rho_\eps\|_{L^1(C(m_\eps,1))}},\] where \[C(r,R):=\{h\in\G:r\leq N(h)\leq R\}\quad\text{for every }0<r<R.\]     It is now clear that $(\bar\rho_\eps)_\eps$ satisfies assumptions \eqref{P1},\eqref{P2},\eqref{P3},\eqref{P5}. In order to verify \eqref{P4}, notice that for every $\delta>0$,
     \[\int_{B^c(0,\delta)}\bar \rho_\eps(h)\,dh\leq\frac{1}{\|\widehat\rho_\eps\|_{L^1(C(m_\eps,1))}}\int_{B^c(0,\delta)}\widehat\rho_\eps\,dh\leq 2\int_{B^c(0,\delta)}\widehat\rho_\eps\stackrel{\eps\to 0}\longrightarrow 0.\]
     A similar argument shows that, recalling also \eqref{bound},
 \[\begin{split}\int_\G\frac{|f(x\cdot h)-f(x)|}{N(h)}|\nabla_\G N(h)|\,\bar\rho_\eps(h)\,dh&\leq 2\int_\G\frac{|f(x\cdot h)-f(x)|}{N(h)}|\nabla_\G N(h)|\,\widehat\rho_\eps(h)\,dh\\ &\leq 4\int_\G\frac{|f(x\cdot h)-f(x)|}{N(h)}|\nabla_\G N(h)|\,\rho_\eps(h)\,dh.\end{split}\]
   Combining this estimate with \eqref{boundednessassumption} and Lemma \ref{firstestimate}, we can apply Corollary \ref{characterization} to deduce that $f\in W^{1,p}_\G(\G)$. The strong convergence in $L^p$ of the non-local gradients $V_\eps(f)$ follows from Corollary \ref{approximation}.  
\end{proof}

\section{Characterizations via Taylor approximation and some consequences}\label{section5}

In this section, we discuss another characterization of Sobolev and BV functions defined on a Carnot group, in terms of a first order Taylor approximation. This notion has been introduced by Spector \cite{spector},\cite{Spector16} for functions defined on $\R^n$, in the spirit of Calderón-Zygmund differentiability. 
The notion of $L^p$ differentiability (of order $k$) for a function defined on the Euclidean space was previously introduced by Calderón and Zygmund in \cite{calderon61} and \cite{calderon62}, in relation to the study of local properties of solutions of elliptic PDEs. A function $f$ is said to be $L^p$ differentiable at a point $x$ if it can be approximated in the $L^p$ norm by a linear mapping around $x$. This notion is the appropriate concept for studying fine differentiability properties of Sobolev functions and functions of bounded variation on $\R^n$ (see for instance the monographs \cite{evans}, \cite{ziemer}, as well as the original proofs in \cite{calderon61},\cite{calderon62}). In the context of Carnot groups, Calderón-Zygmund differentiability for real-valued maps is defined naturally in the following way.
\begin{definition}[$L^p$ differentiability]\label{Lpdiff}
  Let $1\leq p<+\infty$. We say that $f:\Omega\subset \G\to \R$ is $L^{p}$ differentiable at $x\in\Omega$ if there exists a $H$-linear map $L:\G\to\R$ such that
      \[\left(\fint_{B(x,r)}|f(y)-f(x)-L( x^{-1}y)|^{p} \,dy\right)^{1/p}=o(r).\]  
\end{definition}

If $f$ is P-differentiable at $x\in\Omega$ (see Definition \ref{pansudifferentiability}), then it is easy to see that $f$ is also $L^p$ differentiable at $x$ for every $1\leq p<+\infty.$ Similarly, it is clear that $L^q$ differentiability implies $L^p$ differentiability if $p\leq q$. Notice also that the previous definition does not depend on the specific homogeneous norm on $\G$. 
\begin{remark}
    The notion of $L^p$ differentiability is strictly related to the rate of convergence of integral averages on smaller and smaller balls. Let us denote by $f_{r,x}:=\fint_{B(x,r)} f(x)\,dx$ \vspace{-0.1 cm}the mean value of $f$ on $B(x,r)$. If $f\in W^{1,p}_\G(\Omega)$ and $x$ is a Lebesgue point of $\nabla_\G f$, then standard arguments give, by Poincaré inequality, that the following conditions are equivalent:
    \begin{itemize}
        \item $|f_{r,x} -f(x)|=o (r)$;
        \item $f$ is $L^p$ differentiable at $x$.
    \end{itemize}
\end{remark}
It is then natural to investigate fine differentiability properties for functions defined on Carnot groups. This has been achieved by Ambrosio and Magnani \cite{AM},\cite{magnani} in the real valued case (also proving higher order results) and later by Vodopyanov \cite{Vodop} and Kleiner, Müller, Xie \cite{KMX} for Sobolev maps between Carnot groups.\medskip\\ For every $1\leq p<Q$, we let $p^\ast$ to be the \textit{Sobolev conjugate} of $p$, which is defined by
\[p^\ast:=\frac{Qp}{Q-p}.\] Here we focus on the scalar case (see \cite[Theorem 2.4]{magnani}). \begin{theorem}\label{lp*differentiability2}
  Let $1\leq p<Q$ and let $f\in W^{1,p}_{\G,\emph  {loc}}(\Omega)$. Then $f$ is $L^{p^*}$ differentiable almost everywhere on $\Omega$. More precisely, for a.e. $x\in \Omega$, it holds
  \begin{equation}\label{lp*diff2}
      \left(\fint_{B(x,r)}|f(y)-f(x)-\langle \nabla_\G f(x), \pi_x( x^{-1}y)\rangle_x|^{p^*} \,dy\right)^{1/p^*}=o(r).
  \end{equation}
  \end{theorem}
 \begin{remark}
      If $f\in W^{1,Q}_{\G, \text{loc}}(\Omega)$, we can still apply this result, inferring that $f$ is $L^p$ \vspace{-0.08cm}differentiable a.e. on $\Omega$ for every $1\leq p<+\infty$. Moreover, if $f\in W^{1,p}_{\G, \rm loc}(\Omega)$ for some $p>Q$, then it is well known that $f$ is actually P-differentiable almost everywhere on $\Omega$, as can be inferred from a Morrey-type inequality \cite{montiserra}. In particular $f\in W^{1,p}_{\G,\rm loc}(\Omega)$ is $L^p$ differentiable a.e. for any $p\geq 1$.
 \end{remark}
  We also recall the corresponding Calderón-Zygmund differentiability result for functions of bounded horizontal variation (see \cite[Theorem 2.2]{AM} and \cite[Theorem 2.3]{magnani}).
\begin{theorem}\label{l1differentiabilityBV2}
Let $f\in BV_{\G,\emph  {loc}}(\Omega)$. Then $f$ is $L^{1^*}$ differentiable (hence also $L^1$ differentiable) almost everywhere on $\Omega$. More precisely, for a.e. $x\in \Omega$, it holds
  \begin{equation}\label{l1*diff2}
      \left(\fint_{B(x,r)}|f(y)-f(x)-\langle \nabla^{\rm ac}_\G f(x), \pi_x( x^{-1}y)\rangle_x|^\frac{Q}{Q-1} \,dy\right)^\frac{Q-1}{Q}=o(r).
  \end{equation}
  \end{theorem}
It is clear that $L^p$ differentiability does not characterize Sobolev functions. However, if we consider the corresponding \textit{global} condition, this is actually true. As anticipated, the result is due to Spector \cite{spector}, \cite{Spector16} by means of the use of Taylor approximations. 
The definition in the setting of Carnot groups is a natural extension of the Euclidean one. We give here a more flexible definition, which is formulated in terms of general mollifiers (see \cite{brezis-nguyen}, \cite{Ponce}).
\begin{definition}[$L^p$-Taylor approximation]
Let $(\rho_\eps)_\eps$ be a family of mollifiers on $\G$ satisfying assumptions $\eqref{P1}\div\eqref{P4}$. A function $f\in L^p(\G)$ is said to have a first order $L^p$-Taylor approximation with respect to the family $(\rho_\eps)_\eps$ if there exists a section $v\in L^p(\G, H\G)$ such that 
\[\lim_{\eps\to 0}\int_\G\int_{\G}\frac{|f(x\cdot h)-f(x)-\langle v(x), \pi_x(h)\rangle_x|^p}{N(h)^p}\rho_\eps(h)\,dh\,dx=0.\]
\end{definition}
\noindent Notice that, taking $\rho_\eps(h):=\chi_{B(0,\eps)}(h)/|B(0,\eps)|$ (Example \ref{Examples}), the pointwise condition 
\[\lim_{\eps\to 0}\fint_{B(0,\,\eps)}\frac{|f(x\cdot h)-f(x)-\langle v(x), \pi_x(h)\rangle_x|^p}{N(h)^p}\,dh=0\]
is equivalent, by \cite[Lemma 1.16]{magnani}, to the $L^p$ differentiability of $f$ at $x$. Hence, the $L^p$-Taylor approximation can be interpreted as the corresponding global condition and, in fact, it implies $L^p$ differentiability almost everywhere (see Remark \ref{approxtodiff}).\medskip\\
Before stating the main results, we recall this classical lemma, which encodes the fact that group translations are continuous in $L^p(\G)$.
\begin{lemma}\label{translations}
    Let $g\in L^p(\G)$. Then 
    \[\lim_{w\to 0}\int_\G|g(x\cdot w)-g(x)|^p\,dx=0.\]
\end{lemma}
\begin{proof}
\noindent The proof is no more difficult than the Euclidean counterpart (see \cite[Proposition 8.5]{folland99}): one can first show the conclusion for $g\in C^\infty_c(\G)\subset \text{Lip}_{\G,c}(\G)$ and then argue by a standard approximation argument.
\end{proof}

As in the Euclidean case, we shall prove that $L^p$-Taylor approximations characterize the Sobolev space $W^{1,p}_\G(\G)$ (Theorem \ref{Taylormollgen}). However, we start our analysis by considering the case of $f\in BV_\G(\G)$: differently from the positive Calderón-Zygmund differentiability result (Theorem \ref{l1differentiabilityBV2}), a $L^1$-Taylor approximation may fail to hold for functions of bounded horizontal variation. Nevertheless, we can still 
 prove the following upper bound in terms of the total mass of the singular part of the distributional gradient. 
\begin{theorem}\label{TaylorBV}
 Let $(\rho_\eps)_\eps$ be a family of mollifiers satisfying $\eqref{P1}\div\eqref{P4}$. Let $f\in BV_\G(\G)$. Then 
  \begin{equation}\label{taylorapproximationbv}
  \limsup_{\eps\to 0}\int_\G\int_{\G}\frac{|f(x\cdot h)-f(x)-\langle\nabla^{\rm ac}_\G f(x), \pi_x(h)\rangle_x|}{N(h)}\rho_\eps(h)\,dh\,dx\leq C |D_\G^sf|(\G),
  \end{equation}
  where $C$ is an absolute constant depending only on the norm $N$.
\end{theorem}
\begin{proof}
The proof follows \cite[Theorem 1.4]{Ponce}. Let $(\eta_\delta)_\delta$ be a family of standard mollifiers on $\G$. For any function $u$ on $\G$, we let $u_\delta:=\eta_\delta\ast u$ and, similarly, if $\mu$ is a Radon measure on $\G$, we write $\mu_\delta:=\eta_\delta\ast \mu$.
For any $x\in\G$ and $\delta>0$, define $f_{x,\delta}\in C^1_\G(\G)$ to be the map given by
\[f_{x,\delta}(y):=f_\delta(y)-\langle(\nabla^{\rm ac}_\G f)_\delta(x),\pi_x(x^{-1}y)\rangle_x.\]
For every $h\in\G$, let $\gamma_h:[0,1]\to\G$ be a constant speed geodesic (with respect to the Carnot-Carathéodory distance) connecting 0 and $h$, as in the proof of Lemma \ref{thirdestimate}. By \eqref{ftc}, we can write, for every $x\in\G, h\in\G$, $\delta>0$, 
\[|f_{x,\delta}(x\cdot h)-f_{x,\delta}(x)|\lesssim N(h)\int_0^1|\nabla_\G f_{x,\delta}(x\cdot \gamma_h(t))|\,dt.\]
Exploiting the definition of $f_{x,\delta}$, we get
\[|f_\delta(x\cdot h)-f_\delta(x)-\langle(\nabla^{\rm ac}_\G f)_\delta(x),\pi_x(h)\rangle_x|\lesssim N(h)\int_0^1|\nabla_\G f_\delta(x\cdot \gamma_h(t))-(\nabla^{\rm ac}_\G f)_\delta(x)|\,dt\]
for every $x,h\in\G$, $\delta>0$. Recalling that, by Lemma \ref{propconv2} (iv), \[\nabla_\G f_\delta=(D_\G f)_\delta=(D^s_\G f)_\delta+(\nabla^{\rm ac}_\G f)_\delta,\] an integration on $\G$ yields
\begin{equation}\label{approximationestimates}\begin{aligned}\int_\G\frac{|f_\delta(x\cdot h)-f_\delta(x)-\langle(\nabla^{\rm ac}_\G f)_\delta(x),\pi_x(h)\rangle_x|}{N(h)}dx\lesssim &\int_0^1\int_\G|(\nabla^{\rm ac}_\G f)_\delta(x\cdot \gamma_h(t))-(\nabla^{\rm ac}_\G f)_\delta(x)|dx\,dt\\ &+\int_0^1\int_\G|(D^s_\G f)_\delta(x\cdot \gamma_h(t))|dx\,dt\end{aligned}\end{equation}
for any $\delta>0$ and for any $h\in\G$. Now notice that, by the right-invariance of the Lebesgue measure and Lemma \ref{propconv2} (i),
\[\int_0^1\int_\G|(D^s_\G f)_\delta(x\cdot \gamma_h(t))|dx\,dt=\int_\G|(D^s_\G f)_\delta(y)|dy\leq |D^s_\G f|(\G)
.\]
Hence, passing to the limit as $\delta\to 0$ in \eqref{approximationestimates}, by Proposition \ref{young}, Lemma \ref{propconv1} and Lebesgue's dominated convergence theorem, we get
\begin{small}\begin{equation*}\int_\G\frac{|f(x\cdot h)-f(x)-\langle\nabla^{\rm ac}_\G f(x),\pi_x(h)\rangle_x|}{N(h)}dx\lesssim \int_0^1\int_\G|\nabla^{\rm ac}_\G f(x\cdot \gamma_h(t))-\nabla^{\rm ac}_\G f(x)|dx\,dt+|D^s_\G f|(\G)\end{equation*}\end{small}
\hspace{-0.12 cm}for any $h\in\G$. 
\\
If $\nabla^{\rm ac}_\G f(x)=\,0$ a.e. $x\in\G$, we are done by the previous estimates. Thus we can assume that $\|\nabla_\G^{\rm ac} f\|_{L^1(\G)}>0$. 
Let $\eta>0$ be fixed. From Lemma \ref{translations} there exists $\sigma>0$ such that
\[\int_\G|\nabla^{\rm ac}_\G f(x\cdot w)-\nabla^{\rm ac}_\G f(x)| dx\leq \frac{\eta}{2}\] for every $w\in B(\sigma)$. By equivalence of the norms $\|\cdot\|_c$ and $N$, let $\beta>0$ be such that $\gamma_h(t)\in B(\sigma)$ for every $h\in B(\beta), \,t\in[0,1]$. It follows that 
\[\int_0^1\int_\G|\nabla^{\rm ac}_\G f(x\cdot \gamma_h(t))-\nabla^{\rm ac}_\G f(x)|dx\,dt\leq \frac{\eta}{2}\quad\text{for every }h\in B(\beta).\]
Note also that, for every $h\in\G$ and any $t\in[0,1]$, 
\[\begin{split}\int_\G|\nabla^{\rm ac}_\G f(x\cdot \gamma_h(t))-\nabla^{\rm ac}_\G f(x)|dx&\leq\int_\G|\nabla^{\rm ac}_\G f(x\cdot \gamma_h(t))|dx\,+\int_\G|\nabla^{\rm ac}_\G f(x)|dx\\ &\leq \int_\G|\nabla^{\rm ac}_\G f(y)|dy\,+\int_\G|\nabla^{\rm ac}_\G f(x)|dx\\&=2\|\nabla^{\rm ac}_\G f\|_{L^1(\G)}.\end{split}\]
By property \eqref{P4}, we find $\bar\varepsilon>0$ such that, for every $\eps<\bar\varepsilon$ \[\int_{\G\setminus B(\beta)}\rho_\eps(h)\,dh\leq \frac{\eta}{4\|\nabla^{\rm ac}_\G f\|_{L^1(\G)}}.\]
Combining the previous estimates together, we get
\[\begin{split}&\int_\G\int_\G\frac{|f(x\cdot h)-f(x)-\langle\nabla_\G f(x),\pi_x(h)\rangle_x|}{N(h)}dx\,\rho_\eps(h)dh\\&=\int_{B(\beta)}\int_\G\frac{|f(x\cdot h)-f(x)-\langle\nabla_\G f(x),\pi_x(h)\rangle_x|}{N(h)}dx\,\rho_\eps(h)dh\\ &\phantom{aa}+\int_{\G\setminus B(\beta)}\int_\G\frac{|f(x\cdot h)-f(x)-\langle\nabla_\G f(x),\pi_x(h)\rangle_x|}{N(h)}dx\,\rho_\eps(h)dh\\ &\lesssim \int_{B(\beta)}\left[\frac{\eta}{2}+|D^s_\G f|(\G)\right]\rho_\eps(h)\,dh+\int_{\G\setminus B(\beta)}\left[2\|\nabla^{\rm ac}_\G f\|_{L^1(\G)}+|D^s_\G f|(\G)\right]\rho_\eps(h)\,dh\\ &\leq |D^s_\G f|(\G)+\frac{\eta}{2}+2\|\nabla^{\rm ac}_\G f\|_{L^1(\G)}\frac{\eta}{4\|\nabla^{\rm ac}_\G f\|_{L^1(\G)}}=|D^s_\G f|(\G)+\eta.\end{split}\]
for every $\eps<\bar\varepsilon$. This gives the desired estimate, concluding the proof.
\end{proof}
 Using a similar argument, we can now prove that horizontal Sobolev functions admit a first order Taylor approximation (see also \cite[Proposition 1]{brezis-nguyen} for an alternative approach).
\begin{theorem}\label{Taylormollgen}Let $(\rho_\eps)_\eps$ be a family of mollifiers satisfying $\eqref{P1}\div\eqref{P4}$.
  Let $1\leq p<+\infty$ and let $f\in W^{1,p}_\G(\G)$. Then $f$ admits a first-order $L^p$-Taylor approximation with respect to the family $(\rho_\eps)_\eps$. More precisely
  \begin{equation}\label{taylorapproximationmoll}
  \lim_{\eps\to 0}\int_\G\int_\G\frac{|f(x\cdot h)-f(x)-\langle\nabla_\G f(x), \pi_x(h)\rangle_x|^p}{N(h)^p}\rho_\eps(h)\,dh\,dx=0.
  \end{equation}
\end{theorem}
\begin{proof} 
The proof can be carried out as the one of Theorem \ref{TaylorBV}, so we omit the details. Notice that in this case $|D^s_\G f|(\G)=0$, which leads to the expected result. One has to take into account the exponent $p$, but this can be handled simply using Jensen's inequality.
\end{proof}
\begin{remark}\label{rmkbarb2} 
Note that \eqref{taylorapproximationmoll} can be seen as the first-order refinement of Barbieri's result \eqref{barbieribbm} and, by contrast, it applies to \textit{any} homogeneous norm on $\G$. Moreover, if $N$ is assumed to be invariant under horizontal rotations, then it is easy to see that \eqref{taylorapproximationmoll} actually implies \eqref{barbieribbm}, providing in this way an alternative proof of such a result (see Corollary \ref{barbiericorollary}). 
\end{remark}

We now state the converse implication of Theorem \ref{Taylormollgen}, showing that, for $1<p<+\infty$, the $L^p$-Taylor approximation characterizes the horizontal Sobolev space $W^{1,p}_\G(\G)$. 
 Notice that this approach does not cover the limit case $p=1$, which will be treated separately using the approximation results of Section \ref{section4}. 
\begin{theorem}\label{converse}
    Let $1<p<+\infty$ and let $f\in L^p(\G)$. Let $(\rho_\eps)_\eps$ be a sequence of mollifiers satisfying assumptions \eqref{P1}$\div$\eqref{P4}. Suppose $f$ admits a 1st order $L^p$-Taylor approximation with respect to the family $(\rho_\eps)_\eps$, that is there exists a section $v\in L^p(\G,H\G)$ such that
    \[\lim_{\eps\to 0}\int_\G\int_{\G}\frac{|f(x\cdot h)-f(x)-\langle v(x), \pi_x(h)\rangle_x|^p}{N(h)^p}\rho_\eps(h)\,dh\,dx=0.\]
    Then $f\in W^{1,p}_\G(\G)$ and in particular $v(x)=\nabla_\G f(x)$ for a.e. $x\in\G$. 
\end{theorem}
\begin{proof}
 By the triangular inequality, we can estimate
  \[\begin{split}
  &\int_\G\int_\G\frac{|f(x\cdot h)-f(x)|^p}{N(h)^p}\rho_\eps(h)\,dh\,dx\\ &\leq C\underbrace{\int_\G\int_{\G}\frac{|f(x\cdot h)-f(x)-\langle v(x), \pi_x(h)\rangle_x|^p}{N(h)^p}\rho_\eps(h)\,dh\,dx}_{(I)}+C\underbrace{\int_\G\int_{\G}\frac{|\langle v(x), \pi_x(h)\rangle_x|^p}{N(h)^p}\rho_\eps(h)\,dh\,dx}_{(II)}.
  \end{split}\]
  By assumption, quantity $(I)$ is bounded, at least for $\eps$ sufficiently small. The same holds for the term $(II)$, since 
  \[\frac{|\langle v(x), \pi_x(h)\rangle_x|^p}{N(h)^p}\leq |v(x)|_x^p\frac{|\pi_x(h)|_x^p}{N(h)^p}\leq C|v(x)|_x^p,\]
where the last estimate follows for instance by a comparison with the infinity distance $d_\infty$ (see \cite[Theorem 5.1]{FSSC03}) and the fact that any two homogeneous norms are equivalent (explicitly, $|\pi_x(h)|_x\leq d_\infty(h,0)\leq CN(h)$).
  Integrating over $\G$, this implies that $(II)\leq C\|v\|_{L^p(\G)}^p<+\infty$.
  Hence
\begin{equation}\label{limsup}\limsup_{\eps\to 0}\int_\G\int_{\G}\frac{|f(x\cdot h)-f(x)|^p}{N(h)^p}\rho_\eps(h)\,dh\,dx<+\infty.\end{equation}
 We can therefore apply Theorem \ref{firstcharacterization}, inferring from \eqref{limsup} that $f\in  W^{1,p}_\G(\G)$. 
  We finally show that $v(x)=\nabla_\G f(x)$ for a.e. $x\in\G$. This is a standard argument:
 \begin{small} \[\begin{split}
     \int_\G\int_{\G}\frac{|\langle \nabla_\G f(x)-v(x), \pi_x(h)\rangle_x|^p}{N(h)^p}\rho_\eps(h)\,dh\,dx &\leq C \int_\G\int_{\G}\frac{|f(x\cdot h)-f(x)-\langle v(x), \pi_x(h)\rangle_x|^p}{N(h)^p}\rho_\eps(h)\,dh\,dx\\  &\quad+ C\int_\G\int_{\G}\frac{|f(x\cdot h)-f(x)-\langle \nabla_\G f(x), \pi_x(h)\rangle_x|^p}{N(h)^p}\rho_\eps(h)\,dh\,dx.
  \end{split}\]
  \end{small}The right-hand side goes to 0, as $\eps\to 0$, by our assumption and by Theorem \ref{Taylormollgen}. The left-hand side, instead, does not depend on $\varepsilon$: this can be seen using homogeneity and polar coordinates (Theorem \ref{polarcoordinates}). More precisely, if $a(x):=\nabla_\G f(x)-v(x)$, we compute
\[\begin{split}\int_{\G}\frac{|\langle a(x), \pi_x(h)\rangle_x|^p}{N(h)^p}\rho_\eps(h)\,dh  &= \int_0^{+\infty}\int_{S_\G}\frac{|\langle a(x), \pi_x(\delta_r(y))\rangle_x|^p}{N(\delta_r(y))^p}\rho_\eps(\delta_r(y))\,r^{Q-1}d\sigma(y)\,dr\\ &=\int_0^{+\infty}\int_{S_\G}\frac{|\langle a(x), r\,\pi_x(y)\rangle_x|^p}{r^p}\widetilde \rho_\eps(r)\,r^{Q-1}d\sigma(y)\,dr\\ &=\int_0^{+\infty}\widetilde \rho_\eps(r)r^{Q-1}dr\int_{S_\G}|\langle a(x), \pi_x(y)\rangle_x|^p d\sigma(y)\\ &=\underbrace{\int_\G \rho_\eps(h)\,dh}_{1}\fint_{S_\G}|\langle a(x), \pi_x(y)\rangle_x|^p d\sigma(y)
,\end{split}\]
which does not depend on $\eps$. Hence we conclude that
\[\int_\G\int_{\G}\frac{|\langle \nabla_\G f(x)-v(x), \pi_x(h)\rangle_x|^p}{N(h)^p}\rho_\eps(h)\,dh\,dx=0,\]which also implies
\[\int_{\G}\frac{|\langle \nabla_\G f(x)-v(x), \pi_x(h)\rangle_x|^p}{N(h)^p}\rho_\eps(h)\,dh=0 \quad\text{ for a.e. }x\in\G.\]
This is only possible if $\nabla_\G f(x)=v(x)$ for a.e. $x\in \G$.
\end{proof}
Let us now consider the case of the first Sobolev space $W^{1,1}_\G(\G)$. In this case, the above argument would only tell us that $f\in BV_\G(\G)$. For this reason, we are going to apply the results of the previous section to show that the distributional horizontal gradient of $f$ is actually represented by a $L^1$ function. Instead, the argument proposed in \cite[Theorem 1.3]{spector} does not apply in our case, due to the possible lack of symmetry of the horizontal coordinates in a general Carnot group.

\begin{theorem}\label{converse2}
    Suppose $f\in L^1(\G)$ admits a 1st order $L^1$-Taylor approximation, that is there exists a section $v\in L^1(\G,H\G)$ such that
    \begin{equation}\label{assumptionL1}\lim_{\eps\to 0}\int_\G\int_{\G}\frac{|f(x\cdot h)-f(x)-\langle v(x), \pi_x(h)\rangle_x|}{N(h)}\rho_\eps(h)\,dh\,dx=0.\end{equation}
    Then $f\in W^{1,1}_\G(\G)$ and $v(x)=\nabla_\G f(x)$ for a.e. $x\in\G$. 
\end{theorem}
\begin{proof}
    Proceeding as in the proof of Theorem \ref{converse}, we get 
    \[\limsup_{\eps\to 0}\int_\G\int_\G\frac{|f(x\cdot h)-f(x)|}{N(h)}\rho_\eps(h)\,dh\,dx<+\infty.\]
    We can therefore apply Theorem \ref{firstcharacterization} (or, alternatively, \cite[Theorem 1.1]{lathi}) deducing that $f\in BV_\G(\G)$.
    \medskip\\
    Let us denote by $D_\G f$ the distributional horizontal gradient of $f\in BV_\G(\G)$ and let's show that 
    \begin{equation}\label{inteq}\int_\G \varphi\, dD_\G f=\int_\G v\varphi \, d \mathcal L^n\quad \text{ for all }\varphi\in C^0_c(\G).\end{equation}
    Recall from Section \ref{section4} that we can define the nonlocal horizontal gradients 
    \[V_\eps(f)(x):= Q\int_\G \frac{f(x\cdot h)-f(x)}{N(h)}\nabla_\G N(h)\rho_\eps(h)\,dh,\]
    which are well-defined since $f\in BV_\G(\G)$.
    It is now convenient to write
    \[V_\eps(f)(x)=\underbrace{Q\int_\G \frac{f(x\cdot h)-f(x)-\langle v, \pi_x(h)\rangle_x}{N(h)}\nabla_\G N(h)\rho_\eps(h)\,dh}_{I_\eps(x)}\, +\, \underbrace{Q\int_\G\frac{\langle v, \pi_x(h)\rangle_x}{N(h)}\nabla_\G N(h)\rho_\eps(h)\,dh.}_{II_\eps(x)}\]
    We prove that
    \begin{equation}\label{Iepsto0}
        I_\eps\to 0\quad\text{in } L^1(\G)
    \end{equation}
    and
    \begin{equation}\label{IIeps=v}
        II_\eps(x)=v(x)\quad \text{for a.e. }x\in \G.
    \end{equation}
    Indeed
    \[\begin{split}|I_\eps(x)|&\leq Q\||\nabla_\G N|\|_{L^\infty(\G)}\int_\G\frac{|f(x\cdot h)-f(x)-\langle v, \pi_x(h)\rangle_x|}{N(h)}\rho_\eps(h)\,dh
    .\end{split}\]
    Integrating the previous inequality on $\G$ and using \eqref{assumptionL1} we get \eqref{Iepsto0}.
    Let's now prove \eqref{IIeps=v}. 
    For every $\eps>0$, let $(\rho_{\eps,j})_j\subset \rm{Lip}_{\G,c}(\R^n\setminus\{0\})$ be a sequence of maps satisfying \eqref{P1}, \eqref{P2} and such that $\rho_{\eps,j}\to \rho_\eps$ in $L^1(\G)$ as $j\to+\infty$. This sequence can be constructed as in Step 3 of Theorem \ref{convolution}. Then, by Lemma \ref{propertiesofKeps} (iv), arguing as in Step 1 of Theorem \ref{convolution}, it follows that
    \[Q\int_\G\frac{\langle v,\pi_x(h)\rangle_x}{N(h)}\nabla_\G N(h)\,\rho_{\eps,j}(h)\,dh=-\int_\G\langle v,\pi_x(h)\rangle_x \nabla_\G K_{\eps,j}(h)\,dh=v(x)\int_\G K_{\eps,j}(h)\,dh,\]
    where $K_{\eps,j}$ are associated to $\rho_{\eps,j}$ as in \eqref{Keps}. Notice that, as observed in the previous proof, for fixed $x$ the quantity $\frac{\langle v,\pi_x(h)\rangle_x} {N(h)}\nabla_\G N(h)$ is uniformly bounded in $h$. Since $\rho_{\eps,j}\to\rho_\eps$ in $L^1(\G)$ and also $K_{\eps,j}\to K_\eps$ in $L^1(\G)$ (by Step 3 in Theorem \ref{convolution}), 
     we can pass to the limit as $j\to\infty$ and we get
    \[Q\int_\G\frac{\langle v,\pi_x(h)\rangle_x}{N(h)}\nabla_\G N(h)\,\rho_{\eps}(h)\,dh=v(x)\int_\G K_{\eps}(h)\,dh=v(x),\]
    where in the last equality we use (i) in Lemma \ref{propertiesofKeps}, showing \eqref{IIeps=v}.
    \\
    Combining \eqref{Iepsto0} and \eqref{IIeps=v}, we get that $V_\eps(f)\to v$ in $L^1(\G)$. On the other hand, by Corollary \ref{approximation}, it follows that $V_\eps(f)\rightharpoonup D_\G f$ in $\mathcal M(\G;\R^{m_1}).$
    Therefore, for every $\varphi\in C^0_c(\G)$,
    \[\int_\G V_\eps(f)\varphi\,d\mathcal L^n\to\int_\G v\varphi\,d\mathcal L^n\qquad \text{and}\qquad \int_\G V_\eps(f)\varphi\,d\mathcal L^n\to\int_\G\varphi \,dD_\G f,\]
    deducing \eqref{inteq}. Hence $D_\G f=v \,d\mathcal L^n$. In particular there exists $\nabla_\G f=v$ in weak sense and $f\in W^{1,1}_\G(\G),$ completing the proof.
\end{proof}\begin{remark}\label{approxtodiff}
    Combining Theorem \ref{converse} and Theorem \ref{converse2} with Theorem \ref{lp*differentiability2}, we deduce that the $L^p$-Taylor approximation (even with general mollifiers) implies $L^p$ differentiability (actually $L^{p^\ast}$ differentiability) almost everywhere. This is not straightforward, since in general $L^1$ convergence does not imply pointwise convergence a.e., without passing to a subsequence. 
\end{remark}
Notice that the first part of the proof of Theorem \ref{converse2} actually shows the following result for functions in $BV_\G(\G)$.
\begin{theorem}\label{converse2bv} Let $f\in L^1(\G)$. Suppose there exists a section $v\in L^1(\G,H\G)$ such that
\begin{equation}\label{assumptionL1BV}\limsup_{\eps\to 0}\int_\G\int_{\G}\frac{|f(x\cdot h)-f(x)-\langle v(x), \pi_x(h)\rangle_x|}{N(h)}\rho_\eps(h)\,dh\,dx<+\infty.\end{equation}
    Then $f\in BV_\G(\G)$. 
\end{theorem}
We can now summarize the main results of this section, finally deriving Theorem \ref{secondcharBVSobfunct}.
\begin{proof}[Proof of Theorem \ref{secondcharBVSobfunct}]
By collecting all the previous Theorems \ref{TaylorBV}, \ref{Taylormollgen}, \ref{converse}, \ref{converse2}, \ref{converse2bv} together the proof follows.
\end{proof}
\noindent\textbf{Some consequences of the $L^p$-Taylor approximation.}\vspace{0.1 cm}\\
We close this section by pointing out some interesting consequences of Theorem \ref{Taylormollgen}.
Let $f\in W^{1,p}_\G(\G)$ and let $(\rho_\eps)_\eps$ be a family of mollifiers satisfying $\eqref{P1}\div\eqref{P4}$. By Theorem \ref{Taylormollgen}, we know that
\begin{equation}\label{taylorapproximation2}
  \lim_{\eps\to 0}\int_\G\int_\G\frac{|f(x\cdot h)-f(x)-\langle\nabla_\G f(x), \pi_x(h)\rangle_x|^p}{N(h)^p}\rho_\eps(h)\,dh\,dx=0.
  \end{equation}
By Jensen's inequality, this implies that\
\begin{equation*}
\begin{aligned}&\int_\G\Bigg|\underbrace{Q\int_\G\frac{f(x\cdot h)-f(x)}{N(h)}\nabla_\G N(h)\rho_\eps(h)\,dh}_{A_\eps(x)}-\underbrace{Q\int_\G\frac{\langle\nabla_\G f(x), \pi_x(h)\rangle_x}{N(h)}\nabla_\G N(h)\rho_\eps(h)\,dh}_{B_\eps(x)}\Bigg|^pdx\\ &\leq Q\int_\G\int_\G\frac{|f(x\cdot h)-f(x)-\langle\nabla_\G f(x), \pi_x(h)\rangle_x|^p}{N(h)^p}\|\nabla_\G N\|_{L^\infty(\G)}^p\,\rho_\eps(h)\,dh\,dx\stackrel{\eps\to 0}\longrightarrow 0.\end{aligned}\end{equation*}
Notice that $A_\eps$ coincides with the nonlocal horizontal gradients $V_\eps(f)$. Moreover, observe that $B_\eps(x)$ actually does not depend on $\eps$, by integrating in polar coordinates as in Theorem \ref{converse}:
\[\begin{split}B_\eps(x)&:=Q\int_\G\frac{\langle\nabla_\G f(x), \pi_x(h)\rangle_x}{N(h)}\nabla_\G N(h)\rho_\eps(h)\,dh\\ &=Q\int_0^{+\infty}\int_{S_\G}\frac{\langle\nabla_\G f(x), \pi_x(\delta_r(y))\rangle_x}{N(\delta_r(y))}\nabla_\G N(\delta_r(y))\,\rho_\eps(\delta_r(y))\,r^{Q-1}\,d\sigma(y)\,dr\\ &=Q\int_0^{+\infty}\int_{S_\G}{\langle\nabla_\G f(x), \pi_x(y)\rangle_x}\,\nabla_\G N(y)\,\tilde\rho_\eps(r)\,r^{Q-1}\,d\sigma(y)\,dr\\ &=Q \int_0^{+\infty}\tilde\rho_\eps(r)\,r^{Q-1}\,dr \int_{S_\G}{\langle\nabla_\G f(x), \pi_x(y)\rangle_x}\,\nabla_\G N(y)\,\,d\sigma(y)\\ &=Q \fint_{S_\G}{\langle\nabla_\G f(x), \pi_x(y)\rangle_x}\,\nabla_\G N(y)\,\,d\sigma(y)=: B(x).\end{split}\]
Summarizing, the $L^p$-Taylor approximation implies that
\[V_\eps(f)\stackrel{\eps\to 0}\rightarrow B\quad \text{in }L^p(\G).\]
On the other hand, by Corollary \ref{approximation}, we already know that
\[V_\eps(f)\stackrel{\eps\to 0}\rightarrow\nabla_\G f\quad \text{in }L^p(\G).\]
Therefore,
\begin{equation*}\nabla_\G f(x)=Q \fint_{S_\G}{\langle\nabla_\G f(x), \pi_x(y)\rangle_x}\,\nabla_\G N(y)\,\,d\sigma(y)\quad\text{for a.e. }x\in\G.\end{equation*}
Taking $f$ to be locally linear, one gets the following representation formula for a vector $v\in \R^{m_1}$:
\begin{equation}\label{formulaforv}v=Q \fint_{S_\G}{\langle v, \pi_x(y)\rangle_x}\,\nabla_\G N(y)\,\,d\sigma(y).\end{equation}
In the Euclidean case, this expression takes the form
\[v=n \fint_{\mathbb S^{n-1}}{\langle v, y\rangle}\,y\,d\mathcal H^{n-1}(y)\quad\text{ for all $v\in\R^n$},\]
which is a well-known identity following from the high symmetry of $\mathbb S^{n-1}$. 
In the case of a Carnot group endowed with a general homogeneous norm, the same proof does not work anymore and we are not aware of similar reconstruction formulae present in the existing literature 
\\
Even if Theorem \ref{Taylormollgen} is independent of the results of section 3, the above computations show that the $L^p$-Taylor approximation actually implies the convergence of the nonlocal gradients $V_\eps(f)$ to a limit value $B$. If one assumes \eqref{formulaforv}, then this limit coincides with $\nabla_\G f$, providing in this way an alternative proof of Corollary \ref{approximation}. 
\\\\ 
The $L^p$-Taylor approximation directly implies convergence results also for the nonlocal energies $I_{\eps,p}(f)$ and $I^*_{\eps,p}(f)$.
More precisely, from \eqref{taylorapproximation2} and the triangular inequality, computations analogous to the ones above yield for instance
\begin{equation}\label{Iconvergence}\begin{aligned}I^\ast_{\eps,p}(f)=\int_{\G}\int_{\G}\frac{|f(x\cdot h)-f(x)|^p}{N(h)^p}\,\rho_\eps(h)\,dh\,dx\stackrel{\eps\to 0}\longrightarrow&\int_\G \fint_{S_\G}|\langle \nabla_\G f(x),\pi_x(y)\rangle_x|^p\,d\sigma(y)\,dx\\ &=\frac{p+Q}{\sigma(S_\G)}\int_\G\int_{B(0,1)}|\langle \nabla_\G f(x),\pi_x(y)\rangle_x|^p\,dydx,\end{aligned}\end{equation}
where the last equality follows again by an  integration in polar coordinates.\\
This asymptotic relation provides several corollaries: in addition to the classical BBM formula on $\R^n$ \cite{BBM}, we recover for instance the corresponding result by Barbieri for rotationally invariant norms on Carnot groups \cite{barbieri} 
and the limiting behaviour of the anisotropic Gagliardo seminorms established by Ludwig in \cite{ludwig}. Our approach offers an alternative and independent way to derive these results as a direct consequence of the $L^p$-Taylor approximation. 
\begin{corollary}\label{barbiericorollary}\cite[Theorem 3.6]{barbieri}
    Let $f\in W^{1,p}_\G(\G)$ and assume $N$ to be invariant under horizontal rotations. Then 
    \[\lim_{\eps \to 0}\int_\G\int_\G\frac{|f(x\cdot h)-f(x)|^p}{N(h)^p}\rho_\eps(h)\,dh\,dx=C\int_\G|\nabla_\G f(x)|^p\,dx,\]
    with 
\[C:=\fint_{S_\G}\left|\left\langle v,\pi_x(y)\right\rangle_x\right|^p\,d\sigma(y)=\frac{p+Q}{\sigma(S_\G)}\int_{B(1)}|\langle v,\pi_x(y)\rangle_x|^pdy,\]
being $v$ any unit vector in $\R^{m_1}$.
\end{corollary}
\begin{proof}
 By \eqref{Iconvergence}, using the rotational invariance property, we immediately get
 \[\begin{split}\int_{\G}\int_{\G}\frac{|f(x\cdot h)-f(x)|^p}{N(h)^p}\,\rho_\eps(h)\,dh\,dx&\stackrel{\eps\to 0}\longrightarrow\int_\G \fint_{S_\G}|\langle \nabla_\G f(x),\pi_x(y)\rangle_x|^p\,d\sigma(y)\,dx\\ &=\int_\G |\nabla_\G f(x)|^p\fint_{S_\G}\left|\left\langle \frac{\nabla_\G f(x)}{|\nabla_\G f(x)|},\pi_x(y)\right\rangle_x\right|^p\,d\sigma(y)\,dx\\ &=C \int_\G|\nabla_\G f(x)|^p dx,\end{split}\]
 for the constant $C$ as in the statement of the corollary.
\end{proof}
The convergence result for anisotropic Gagliardo seminorms established by Ludwig \cite{ludwig} can be formulated in the setting of general Carnot groups as follows.
\begin{corollary}\label{thmLudwig}
Let $f\in W^{1,p}_\G(\G)$. Then
\begin{equation}\label{fractional}\lim_{\eps\to0} \eps\int_\G\int_\G\frac{|f(x\cdot h)-f(x)|^p}{N(h)^{Q+p-\eps p}}dhdx= \frac{p+Q}{p}\int_\G\int_{B(0,1)}|\langle \nabla_\G f(x),\pi_x(y)\rangle_x|^p\,dydx.\end{equation}
\end{corollary}
\begin{proof}
Let $R>0$ be fixed. Then simple computations give
\[\begin{split} \eps\int_{\G}\int_{B^c(R)}\frac{|f(x\cdot h)-f(x)|^p}{N(h)^{Q+p-\eps p}}dhdx&\leq \eps 2^{p-1}\int_{\G}\int_{B^c(R)}\frac{|f(x\cdot h)|^p+|f(x)|^p}{N(h)^{Q+p-\eps p}}dhdx\\&= \eps 2^{p}\|f\|_{L^p(\G)}^p\int_{B^c(R)}\frac{1}{N(h)^{Q+p-\eps p}}dh\\ &=\eps 2^p\|f\|_{L^p(\G)}^p\sigma(S_\G)\int_R^{+\infty}\rho^{\eps p-p-1}d\rho\\
&=\eps 2^p\|f\|_{L^p(\G)}^p\sigma(S_\G)\frac{R^{\eps p-p}}{p-\eps p}\stackrel{\eps \to 0}\longrightarrow 0.\end{split}\]
Hence \eqref{fractional} is equivalent to compute 
\[\lim_{\eps\to0} \eps\int_\G\int_{B(R)}\frac{|f(x\cdot h)-f(x)|^p}{N(h)^{Q+p-\eps p}}dhdx.\]
Defining \[\rho_\eps(h):=\frac{\eps p\,\chi_{B(0,R)(h)}}{N(h)^{Q-\eps p}\sigma(S_\G)R^{\eps p}},\]
it is easy to verify that $(\rho_\eps)_\eps$ satisfy assumptions $\eqref{P1}\div\eqref{P4}$ (see Example \ref{Examples}). Hence we can apply \eqref{Iconvergence} and get
\begin{align*}
    \lim_{\eps\to0} \eps\int_\G\int_\G\frac{|f(x\cdot h)-f(x)|^p}{N(h)^{Q+p-\eps p}}dhdx&=\frac{\sigma(S_\G)}{p}\lim_{\eps \to 0}\int_\G\int_\G\frac{|f(x\cdot h)-f(x)|^p}{N(h)^p}\rho_\eps(h)\,dhdx\\ &=\frac{\sigma(S_\G)}{p}\frac{p+Q}{\sigma(S_\G)}\int_{B(0,1)}|\langle \nabla_\G f(x),\pi_x(y)\rangle_x|^p\,dy.\\
    &=\frac{p+Q}{p}\int_\G\int_{B(0,1)}|\langle \nabla_\G f(x),\pi_x(y)\rangle_x|^p\,dydx\qedhere
\end{align*}\end{proof}
Corollary \ref{thmLudwig} is a generalization of \cite[Theorem 1]{ludwig}, which instead deals with the case $\G=(\R^n,+)$ equipped with an arbitrary (anisotropic) norm. Notice that our approach works for arbitrary Carnot groups and it does not require the assumption of $f$ to be compactly supported (which is instead assumed in \cite{ludwig}). 

    \bibliographystyle{abbrv}
\bibliography{biblio}
\end{document}